\documentclass[aos]{imsart}

\RequirePackage{amsthm,amsmath,amsfonts,amssymb}
\RequirePackage[authoryear]{natbib}
\RequirePackage[colorlinks,citecolor=blue,urlcolor=blue]{hyperref}
\usepackage{algorithm, algorithmic,enumerate,booktabs}
\usepackage[caption=false]{subfig}
\usepackage{tikz}
\usetikzlibrary{positioning,calc}

\startlocaldefs
\numberwithin{equation}{section}
\theoremstyle{plain}
\newtheorem{thm}{Theorem}[section]
\newtheorem{lem}[thm]{Lemma}
\newtheorem{prop}[thm]{Proposition}
\newtheorem{cor}[thm]{Corollary}
\theoremstyle{remark}
\newtheorem{defn}[thm]{Definition}
\newtheorem{exmp}[thm]{Example}
\newtheorem{rem}[thm]{Remark}
\newcommand{\Left}{\text{Left}}
\newcommand{\Right}{\text{Right}}
\renewcommand{\O}{\Omega}
\renewcommand{\L}{\Lambda}
\newcommand{\Od}{\O_{\mathsf{diag}}}
\newcommand{\Gl}{G^\cL}
\newcommand{\cL}{\mathcal{L}}
\renewcommand{\phi}{\varphi} 
\newcommand{\im}{\text{Im}}
\newcommand{\pa}{\mathrm{pa}}
\newcommand{\ch}{\mathrm{ch}}
\newcommand{\htr}{\mathrm{htr}}
\newcommand{\rank}{\text{rank}}
\newcommand{\Gf}{G_{\textrm{flow}}}
\newcommand{\bR}{\mathbb{R}}
\newcommand{\ol}[1]{\overline{#1}}

\newcommand{\algorithmicbreak}{\textbf{break}}
\newcommand{\BREAK}{\STATE \algorithmicbreak}

\endlocaldefs

\begin{document}
\sloppy 

\begin{frontmatter}
\title{Half-Trek Criterion for Identifiability \\of Latent Variable Models}
\runtitle{Identifiability in Latent Variable Models}

\begin{aug}
\author[A]{\fnms{Rina Foygel} \snm{Barber}\ead[label=e1]{rina@uchicago.edu}},
\author[B]{\fnms{Mathias} \snm{Drton}\ead[label=e2,mark]{mathias.drton@tum.de}},
\author[B]{\fnms{Nils} \snm{Sturma}\ead[label=e3,mark]{nils.sturma@tum.de}}
\and
\author[C]{\fnms{Luca} \snm{Weihs}\ead[label=e4]{lucaw@allenai.org}}
\address[A]{Department of Statistics, University of Chicago, \printead{e1}}
\address[B]{Department of Mathematics and Munich Data Science Institute, Technical University of Munich, \printead{e2,e3}}
\address[C]{Allen Institute for AI, \printead{e4}}
\end{aug}

\begin{abstract}
We consider linear structural equation models with latent variables and develop a criterion to certify whether the direct causal effects between the observable variables are identifiable based on the observed covariance matrix.  Linear structural equation models assume that both observed and latent variables solve a linear equation system featuring stochastic noise terms.  Each model corresponds to a directed graph whose edges represent the direct effects that appear as coefficients in the equation system.
Prior research has developed a variety of methods to decide identifiability of direct effects in a latent projection framework, in which the confounding effects of the latent variables are represented by correlation among noise terms.  This approach is effective when the confounding is sparse and effects only small subsets of the observed variables.  
In contrast, the new latent-factor half-trek criterion (LF-HTC) we develop in this paper operates on the original unprojected latent variable model and is able to certify identifiability in settings, where some latent variables may also have dense effects on many or even all of the observables.  Our LF-HTC is an effective sufficient criterion for rational identifiability, under which the direct effects can be uniquely recovered as rational functions of the joint covariance matrix of the observed random variables.  When restricting the search steps in LF-HTC to consider subsets of latent variables of bounded size, the criterion can be verified in time that is polynomial in the size of the graph. 
\end{abstract}

\begin{keyword}[class=MSC]
\kwd{62H22}
\kwd{62J05}
\kwd{62R01}
\end{keyword}

\begin{keyword}
\kwd{Covariance matrix}
\kwd{factor analysis}
\kwd{graphical model}
\kwd{hidden variables}
\kwd{latent variables}
\kwd{parameter identification}
\kwd{structural equation model}
\end{keyword}

\end{frontmatter}

\section{Introduction}

Equipped with an intuitive causal interpretation, structural equation models are very popular tools in a broad range of applied sciences \citep{spirtes2000causation,pearl2009causality,peters2017}.  Often, structural equation models involve latent variables, and it becomes a key problem to clarify whether parameters of interest are identifiable from the joint distribution of the observable variables.  Many different criteria have been developed to decide such identifiability.  The dominant approach in  state-of-the-art methods is to project away latent variables, i.e., their effects are absorbed into correlations among error terms in the structural equations.  In contrast, we here consider models with explicit latent variables and show how the latent dependence structure may be used to certify identifiability even in cases with dense latent confounding, where projection approaches remain inconclusive.

Concretely, we study linear structural equation models with explicit latent variables.  The precise setting of interest may be described as follows. Let $X=(X_v)_{v\in V}$ be a collection of $d=|V|$ observed variables, and let $L=(L_h)_{h\in\mathcal{L}}$ be $\ell=|\mathcal{L}|$ latent (unobserved) variables. Suppose all variables are related by linear equations as
$$
    X_v = \sum_{w \neq v} \lambda_{wv} X_w + \sum_{h\in\mathcal{L}} \gamma_{hv} L_h + \varepsilon_v, \qquad v\in V,
$$
where $\lambda_{wv}$ and $\gamma_{hv}$ are real-valued parameters that are also known as \emph{direct causal effects} of $X_w$ on $X_v$ and $L_h$ on $X_v$, respectively.  The $\varepsilon_v$ are independent mean zero random variables that model noise.  We assume that each $\varepsilon_v$ has finite variance $\omega_v > 0$.  The latent variables $(L_h)_{h\in\mathcal{L}}$ are assumed to be independent, and also independent of the noise terms $\varepsilon = (\varepsilon_v)_{v\in V}$.  Since we are primarily interested in identification of direct causal effects $\lambda_{vw}$, we may fix, without loss of generality, the latent scale such that each $L_h$ has mean zero and variance 1. Viewing $X$, $L$, and $\varepsilon$ as vectors,
the above equation system can be presented in the form
\begin{equation}\label{eq:def-model}
    X = \Lambda^{\top} X + \Gamma^{\top} L + \varepsilon
\end{equation}
with $d \times d$ parameter matrix $\Lambda=(\lambda_{wv})$ and $\ell \times d$ parameter matrix $\Gamma=(\gamma_{hv})$. The matrix $\Lambda$ has zeros along the diagonal.  
Specific models are now derived from \eqref{eq:def-model} by assuming specific sparsity patterns in $\Lambda$ and $\Gamma$. The resulting models assume that all unobserved confounding is caused only by the explicitly modeled, independent latent variables. Thus the latent structure corresponds to factor analysis models,  and we will refer to the latent variables also as \emph{latent factors}. 

The models belong to the general framework of structural equation models with latent variables as they are considered, e.g., in \citet{bollen1989structural}. However, where many of the examples in Bollen's book are concerned with measurement models, i.e., latent variables are measured through observations and these observations are conditionally independent given the latent variables, our interest here is the setting where we have  direct causal effects $\lambda_{wv}$ between observed variables and the latent variables constitute confounders.

The focus of this paper will be on the covariance structure posited by models derived from  \eqref{eq:def-model}.  In particular, we will be interested in determining when sparsity in the matrices $\Lambda$ and $\Gamma$ allows one to \emph{identify} (i.e., uniquely recover) the direct effects $\lambda_{wv}$ from the covariance matrix of the observable random vector $X$.  
Solving \eqref{eq:def-model}, we find
\[
X = (I_d - \Lambda)^{-\top}(\Gamma^{\top} L + \varepsilon).
\]
The vector $\Gamma^{\top} L + \varepsilon$ follows a latent factor model and has  covariance matrix
\begin{equation} \label{eq:omega}
\Omega = \text{Var}[ \varepsilon]+\Gamma^{\top}\text{Var}[ L]\Gamma = \Od + \Gamma^{\top} \Gamma = \Od + \sum_{h \in \cL} \Gamma^{\top}_h \Gamma_h,
\end{equation}
where $\Od$ is diagonal with entries $\O_{\mathsf{diag},vv} = \omega_v$ and $\Gamma_h$ is the $h$-th row of $\Gamma$ such that the entries of $\Gamma_h$ correspond to the causal effects associated to the latent factor $L_h$.  We term the matrix $\Omega$ the \emph{latent covariance matrix}.  
 It follows that  $X$ has 
  covariance matrix 
$$
\Sigma = (I_d - \Lambda)^{-\top}\Omega(I_d-\Lambda)^{-1}.
$$

In order to study structural equation models it is useful to adopt a graphical perspective.  To this end, the zero patterns in $\Lambda$ and $\Gamma$ are associated to a directed graph $G = (V\cup\cL, D)$, where  $D \subset (V\cup\cL)\times (V\cup \cL)$ is a collection of directed edges $w \rightarrow v$. For two observed nodes $v,w \in V$, the effect $\lambda_{wv}$ may be nonzero only if the edge $w \rightarrow v$ is contained in the set $D$. Similarly, for a latent node $h \in \cL$ and an observed node $v \in V$, the effect $\gamma_{hv}$ is possibly nonzero if $h \rightarrow v \in D$. In figures we draw latent nodes $h$ in gray, and we draw edges $h \rightarrow v$ dashed for better distinction. This is illustrated in the next example. 

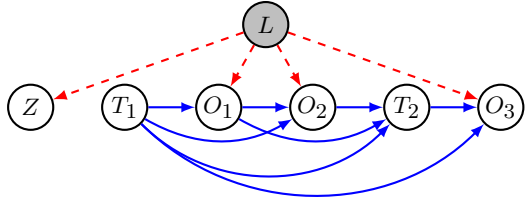
\begin{figure}[t]
    \centering
    \tikzset{
      every node/.style={circle, inner sep=0.3mm, minimum size=0.6cm, draw, thick, black, fill=white, text=black},
      every path/.style={thick}
    }
    \begin{tikzpicture}[align=center]
      \node[fill=lightgray] (l) at (1.875,1.1) {$L$};
      \node[] (t1) at (0,0) {$T_1$};
      \node[] (o1) at (1.25,0) {$O_1$};
      \node[] (o2) at (2.5,0) {$O_2$};
      \node[] (t2) at (3.75,0) {$T_2$};
      \node[] (o3) at (5,0) {$O_3$};
      \node[] (z) at (-1.25,0) {$Z$};
      
      \draw[red, dashed] [-latex] (l) edge (z);
      \draw[red, dashed] [-latex] (l) edge (o1);
      \draw[red, dashed] [-latex] (l) edge (o2);
      \draw[red, dashed] [-latex] (l) edge (o3);
      
      \draw[blue] [-latex] (t1) edge (o1);
      \draw[blue] [-latex] (o1) edge (o2);
      \draw[blue] [-latex] (o2) edge (t2);
      \draw[blue] [-latex] (t2) edge (o3);
      
      \draw[blue] [-latex, bend right] (t1) edge (o2);
      \draw[blue] [-latex, out=-45, in=-135] (t1) edge (t2);
      \draw[blue] [-latex, out=-45, in=-135] (t1) edge (o3);
      \draw[blue] [-latex, bend right] (o1) edge (t2);
    \end{tikzpicture}
    
    \caption{Graph corresponding to a randomized clinical trial for sequential administered treatments with a latent factor $L$.}
    \label{fig:intro-example}
\end{figure}

\begin{exmp}
    We consider an augmented version of an example from \citet[Section 7]{stanghellini2005on}, which pertains to the effects of sequential treatments in randomized clinical trials. Suppose that the patients receive two treatment doses in sequence, $T_1$ and $T_2$, and at both times the treatment dose is assigned at random.  The randomization distribution of the second treatment dose $T_2$ depends on the previous treatment dose $T_1$ and on two intermediate outcome measures $O_1$ and $O_2$. The intermediate outcome measures are deemed potentially related, i.e.,  $O_2$ may causally depend on $O_1$. After the second treatment a final outcome measure $O_3$ is recorded.
    Assume now that there is a latent factor $L$, such as a specific characteristic of a patient, that has effects on all outcomes  $O_1, O_2, O_3$.  Finally, as in \citet{stanghellini2005on}, we assume that there exists an auxiliary observed variable $Z$ that provides a noisy measurement of $L$.  The direct effects in this setup are depicted in the  graph shown in Figure~\ref{fig:intro-example}.  
\end{exmp}

We aim to characterize those models of the form \eqref{eq:def-model} that are rationally identifiable, i.e., all possibly nonzero direct causal effects $\lambda_{wv}$ can be uniquely recovered as rational functions of the entries of the observable covariance matrix $\Sigma$. This kind of identifiability has been examined in previous research in the context of latent projections where latent variables are not explicitly modeled.  Models then correspond to \emph{mixed graphs} that contain only the observed nodes $V$, but bidirectional edges in addition to the directed edges. Each bidirected edge represents a possibly nonzero entry in the latent covariance matrix $\Omega$, i.e., it implicitly indicates the presence of a confounding latent factor. The starting point for deriving sufficient criteria for rational identifiability are then the equations
\begin{equation} \label{eq:old-approach}
    [(I_d - \Lambda)^{\top}\Sigma(I_d-\Lambda)]_{vw} = \Omega_{vw}  = 0,
\end{equation}
which hold whenever no confounding latent factor affects both, $X_v$ and $X_w$ with $v\neq w$.  The equations \eqref{eq:old-approach} are then solved to obtain the nonzero effects in $\Lambda$. This strategy has been leveraged to formulate graphical criteria applicable to mixed graph representations of latent variable models. 

An example of a graphical criterion leveraging the latent projection approach is the half-trek criterion of \cite{foygel2012halftrek}, which can be considered as a predecessor and special case of the new results in this paper. But there are also various other graphical criteria on mixed graphs such as instrumental variables \citep{bowden1984instrumental}, conditional instruments \citep{brito2002generalized}, the $G$-criterion \citep{brito2006graphical}, auxiliary variables (\citeauthor{chen2016incorporating}, \citeyear{chen2016incorporating} and \citeauthor{chen2017identification}, \citeyear{chen2017identification}), decomposition techniques \citep{tian2005identifying} and several generalizations and further developments, cf. \citet{tian2009parameter}, \citet{drton2016generic}, \citet{weihs2017determinantal}, \citet{kumor2019efficient} and \citet{kumor2020efficient}.

In contrast, in this work we consider the original, unprojected latent variable model as defined in \eqref{eq:def-model}, and we allow the latent covariance matrix $\Omega$ to be dense with only few or no zero entries. Then the usual approach of exploiting the zero structure in $\Omega$ that was highlighted in \eqref{eq:old-approach} is no longer effective.  However, dense confounding of the observed variables may be caused by only a small number of latent factors, in which case the latent covariance matrix $\Omega$ exhibits exploitable structure. Our key observation is that $\Omega$ may contain rank-deficient submatrices.  For example, let $Y, Z \subseteq V$ be two disjoint sets  of observed nodes. Then by \eqref{eq:omega} the submatrix $\Omega_{Y,Z}$ equals
$$
\Omega_{Y,Z} =  (\Od)_{Y,Z} + \sum_{h \in \cL} (\Gamma^{\top}_h \Gamma_h)_{Y,Z} = \sum_{h \in H} (\Gamma^{\top}_h \Gamma_h)_{Y,Z},
$$
where the subset $H \subseteq \cL$ over which we sum on the right-hand side contains exactly those latent factors that have an effect on a node in $Y$ and at the same time also an effect on a node in $Z$.
Since the matrix $\Gamma^{\top}_h \Gamma_h$ has rank one for each latent node $h$, the submatrix $\Omega_{Y,Z}$ is not of full column rank if $|H| < |Z|$. Exploiting this low rank structure of the latent covariance matrix $\Omega$ yields our main result, which is a sufficient criterion  for rational identifiability of the direct causal effects $\lambda_{wv}$.  We show how to convert the criterion into a graphical condition that can be checked using efficient algorithms under a bound on the considered rank.  The graphical criterion is directly applicable to directed graphs $G = (V\cup\cL, D)$ that explicitly contain the latent nodes $\cL$, i.e., the criterion operates on the unprojected latent variable model and allows to explore specific confounding.  We refer to it as the \emph{latent-factor half-trek criterion} (LF-HTC).

\begin{exmp}
   We take up the earlier example of a randomized clinical trial with sequential treatments, which we summarized in the graph in Figure~\ref{fig:intro-example}. 
   It is natural to investigate the direct causal effects between the observed variables $T_1, O_1, O_2, T_2$ and $O_3$. These direct causal effects correspond to the blue (non-dashed) edges in the figure.  Our new latent-factor half-trek criterion will be able to certify that the whole parameter matrix $\Lambda$ is rationally identifiable and all nonzero effects $\lambda_{vw}$ can be written as rational formulas in the entries of the observable covariance matrix $\Sigma$. For example, the direct effect from the first treatment dose $T_1$ on the intermediate outcome $O_1$ is given by $\Sigma_{T_1, O_1} / \Sigma_{T_1, T_1}$; a standard regression coefficient. But remarkably, we can even identify effects corresponding to the edges $T_1 \rightarrow O_2$ and $O_1 \rightarrow O_2$ by the latent-factor half-trek criterion. We verified that it is impossible to identify the latter two effects in the latent projection framework (cf. Section \ref{sec:latent-proj}).
   
   
\end{exmp}

While most of the  general identification criteria have been developed in the setting of latent projections, some  existing work also considers unprojected latent factor models as defined in \eqref{eq:def-model}. 
However, this work addresses special types of latent confounding only. 
For example,  \citet{stanghellini2005on} and \citet{leung2016identifiability} examine linear latent variable models with one latent variable, and the conditional instrument approach in \citet{vanderzander2015efficiently} covers scenarios in which no confounding factor has an effect on all observed variables. Another approach requires that latent factors are measured through observed proxy variables and relies on identifying the causal effect between the latent factor and the proxy, see for example \citet{kuroki2014measurement}, \citet{wang2018identifying} and \citet{lee2021causal}, the latter of which deals with the discrete case. 


It should be noted that, in principle, rational identifiability is always decidable by computational algebraic geometry \citep{garcia2010identifying} involving Gr\"obner basis computations \citep{cox2007ideals}. However, in the worst case, the complexity of these methods can be double exponential in the size of the graph. Thus, they may be infeasible even for relatively small graphs, and more efficient graphical criteria are of great value.  To check the new latent-factor half-trek criterion we propose an algorithm based on max-flow computations \citep{cormen2009introduction} that runs in polynomial time in the size of the graph if we confine ourselves to search only over subsets of latent factors of bounded size. We show that the restriction of the search space is necessary since the task of checking the latent-factor half-trek criterion without restrictions is in general NP-complete.  

The organization of the paper is as follows. In Section \ref{sec:graph-rep-and-ident} we provide a precise definition of linear structural equation models given by directed graphs and rigorously introduce the concept of rational identifiability. Moreover, we derive basic necessary conditions for rational identifiability based on dimension arguments. In Section \ref{sec:main-result} we present our main result, the LF-HTC. In Section \ref{sec:latent-proj} we discuss the latent projection framework considered in previous research and compare the new LF-HTC to existing criteria. In particular, we compare the LF-HTC to the original half-trek criterion. In Section \ref{sec:computation} we present an algorithm to check the LF-HTC efficiently. Using this algorithm we systematically check identifiability of certain classes of small latent-factor graphs in Section \ref{sec:experiments}.  The restriction to small graphs allows for these checks to be validated using suitably designed Gr\"obner basis computations.   Finally, the proof of the main result is given in Section \ref{sec:proof-lfhtc}.  Further elements of proofs, a hardness result for checking the LF-HTC without a bound on the cardinality of searched sets of latent variables and an explanation on how to effectively deploy techniques from computational algebraic geometry are deferred to the Supplementary Material \citep{supplemental}.

\section{Graphical Representation and Identifiability} \label{sec:graph-rep-and-ident}
Let $G = (V\cup\cL, D)$ be a directed graph where  $V$ and $\cL$ are finite disjoint sets of observed and latent nodes, respectively.  We emphasize that $G$ is allowed to contain directed cycles.  Let $d=|V|$ and $\ell=|\cL|$.  The edge set $D \subset (V\cup\cL)\times (V\cup \cL)$ 
is assumed to be free of self-loops, so $v \rightarrow v \notin D$ for all $v \in V\cup\cL$. For each vertex $v \in V\cup\cL$, define its set of parents as $\pa(v) = \{w \in V\cup\cL: w \rightarrow v \in D\}$. Throughout the paper we require $\pa(h)=\emptyset$ for all $h\in \cL$, so that all latent nodes are source nodes and the outgoing edges of latent nodes only point to observed nodes. If this condition is satisfied, we call $G$ a \emph{latent-factor graph} and, to emphasize the set of latent variables, write $\Gl$ instead of $G$.

The edge set of a latent-factor graph may be partitioned as $D = D_V \cup D_{\cL V}$, where $D_V = D \cap (V \times V)$ is the set of directed edges between observed nodes and $D_{\cL V} = D \cap (\cL \times V)$ is the set of directed edges that point from latent to observed nodes. Let $\mathbb{R}^{D_V}$ be the set of real $d \times d$ matrices $\Lambda = (\lambda_{wv})$ with support $D_V$, that is, $\lambda_{wv} = 0$ if $w \rightarrow v \notin D_V$. Write $\mathbb{R}^{D_V}_{\mathrm{reg}}$ for the subset of matrices $\Lambda \in \mathbb{R}^{D_V}$ with $I_d - \Lambda$  invertible;  recall that we allow $G^\cL$ to contain directed cycles. Similarly, let $\mathbb{R}^{D_{\cL V}}$ be the set of real $\ell \times d$ matrices $\Gamma = (\gamma_{hv})$ with support $D_{\cL V}$, that is, $\gamma_{hv} = 0$ if $h \rightarrow v \notin D_{\cL V}$. Additionally, we write $\mathrm{diag}^{+}_d$ for the set of all $d \times d$ diagonal matrices with a positive diagonal indexed by the elements of $V$.

Each latent-factor graph postulates a covariance model that corresponds to a linear structural equation model specified via \eqref{eq:def-model}.

\begin{defn} \label{def:model-def}
 The covariance model given by a latent-factor graph $\Gl = (V\cup\cL, D)$ with $|V|=d$ and $|\cL|= \ell$ is the family of covariance matrices
  \begin{equation} \label{eq:param-sigma}
      \Sigma = (I_d - \Lambda)^{-\top}\Omega(I_d-\Lambda)^{-1}
  \end{equation}
  obtained from choices of $\Lambda \in \mathbb{R}^{D_V}_{\mathrm{reg}}$ and $\Omega$ in the image of the map
  \begin{align*}
    \tau : \mathbb{R}^{D_{\cL V}} \times  \mathrm{diag}^{+}_d&\longrightarrow \textit{PD}(d) \\
    (\Gamma, \Od) &\longmapsto \Od + \Gamma^{\top} \Gamma,
\end{align*}
where $\textit{PD}(d)$ is the cone of positive definite symmetric $d \times d$ matrices.  We term the image $\im(\tau) \subseteq \textit{PD}(d)$ the \emph{cone of latent covariance matrices}.
\end{defn}

We are interested in the question of identifiability, i.e., whether the matrix $\Lambda$ can be uniquely recovered from a given covariance matrix $\Sigma$ of the form \eqref{eq:param-sigma}. If it is possible to recover the whole matrix $\Lambda$ uniquely, we can determine $\Omega$ uniquely by the equation
\begin{equation}
\label{eq:Omega-start}
    (I_d - \Lambda)^{\top}\Sigma(I_d-\Lambda) = \Omega,
\end{equation}
since the matrix $I_d - \Lambda$ is assumed to be invertible. Thus, for $\Theta = \mathbb{R}^{D_V}_{\mathrm{reg}} \times \im(\tau)$, identifiability holds if the parametrization map
\begin{align} \label{eq:param-map}
\begin{split}
    \phi_{\Gl} : \Theta &\longrightarrow \textit{PD}(d) \\
    (\Lambda, \Omega) &\longmapsto (I_d - \Lambda)^{-\top}\Omega(I_d-\Lambda)^{-1} 
\end{split}
\end{align}
is injective on $\Theta$, or a suitably large subset. Since identifiability will usually not hold on the whole set $\Theta$, we need to clarify what we mean by a ``suitably large'' subset. We use terminology from algebraic geometry, background can be found in \citet{cox2007ideals}, \citet{shafarevich2013basic} or \citet{hartshorne1977algebraic}.

A property on an irreducible algebraic set $W$ is said to be generically true if the property holds on the complement $W \setminus A$ of a proper algebraic subset $A \subseteq W$.  Due to irreducibility, the complement $W \setminus A$ is dense in $W$ with respect to the Zariski topology and therefore considered as a ``suitably large'' subset. When $W$ is an irreducible algebraic set defined over the real numbers, a proper algebraic subset of $W$ has Lebesgue measure zero, see e.g.~the lemma in \citet{okamoto1973distinctness}.


To connect this terminology to our setup, we observe that the Zariski closure $\overline{\Theta}$, i.e., the smallest algebraic subset that contains the domain $\Theta$, is irreducible. This is true because $\Theta$ is the polynomial image of an open set.  Hence, we say that a property on $\Theta$ is generically true if there exists a proper algebraic subset $A \subset \overline{\Theta}$ such that the property holds on the complement $\Theta \setminus A$. Our interest is now in generically identifying the direct causal effects $\lambda_{wv}$. Since the parametrization $\phi_{G^{\mathcal{L}}}$ is rational, the identification formula, in the worst case, is an algebraic function \citep{garcia2010identifying}. However, in all examples we know, if generic identifiability is possible, then by rational formulas. This motivates the following definition.

\begin{defn}[Rational identifiability] \label{def:rat-ident}
\mbox{ }
\begin{itemize}
    \item[(a)] The latent-factor graph $\Gl$ is said to be rationally identifiable if there exists a proper algebraic subset $A \subset \overline{\Theta}$ and a rational map $\psi : \textit{PD}(d) \longrightarrow \mathbb{R}^{D_V}_{\mathrm{reg}} \times \textit{PD}(d)$ such that $\psi \circ \phi_{\Gl}(\Lambda, \Omega) = (\Lambda, \Omega)$ for all $(\Lambda, \Omega) \in \Theta \setminus A$.
    \item[(b)] 
    The direct causal effect $\lambda_{vw}$, or also simply the edge $v \rightarrow w \in D_V$, is rationally identifiable if there exists a proper algebraic subset $A \subset \overline{\Theta}$ and a rational map $\psi : \textit{PD}(d) \longrightarrow \mathbb{R}$ such that $\psi \circ \phi_{\Gl}(\Lambda, \Omega) = \lambda_{vw}$ for all $(\Lambda, \Omega) \in \Theta \setminus A$.
\end{itemize}
\end{defn}

Rational identifiability of $\Gl$ is equivalent to rational identifiability of all edges in $D_V$; recall~\eqref{eq:Omega-start}.
If $\Gl$ is rationally identifiable, then a (absolutely continuous) random choice of the effects in $(\Lambda,\Gamma)$ and the error variances in $\Od$ will almost surely yield a covariance matrix for the observable vector $X$ from which $\Lambda$ can be recovered uniquely by rational formulas. If $\Gl$ is not generically identifiable, its parametrization $\phi_{\Gl}$ may be either generically finite-to-one or generically infinite-to-one:

\begin{defn}
    Let $f: S \rightarrow \mathbb{R}^n$ be a map defined on a subset $S \subseteq \mathbb{R}^m$ such that the Zariski closure $\overline{S}$ is irreducible. Then $f$ is generically finite-to-one if there exists a proper algebraic subset $A \subseteq \overline{S}$ such that the fiber $\mathcal{F}_{f}(s) = f^{-1}(f(s))$ is finite for all $s \in S \setminus A$. Otherwise, $f$ is said to be generically infinite-to-one.
\end{defn}

\begin{defn} \label{def:finite-to-one}
  A latent-factor graph $\Gl$ is generically finite-to-one if its parametrization $\phi_{\Gl}$ is generically finite-to-one. In this case we will also say that $\Gl$ is finitely identifiable. Otherwise,  $\Gl$ is said to be generically infinite-to-one.
\end{defn}

Note that if a latent-factor graph $\Gl$ is rationally identifiable, then the fiber $\mathcal{F}_{\phi_{\Gl}}(\Lambda, \Omega) = \{(\Lambda, \Omega)\}$ for all parameter choices outside of a proper algebraic subset. In particular, a graph that is rationally identifiable is generically finite-to-one. The following Lemma is an important tool to check if a rational map is generically finite-to-one. For completeness, we provide a proof in Appendix A in the supplement \citep{supplemental}.  Here, we rely on the notion of semialgebraic sets, which are finite unions of sets defined by
finitely many  polynomial equations and inequalities. For background on semialgebraic sets we refer to \citet{bochnak1998real}, \citet{basu2006algorithms} and \citet{benedetti1990real}.

\begin{lem} \label{lem:finite-to-one}
    Let $S \subseteq \mathbb{R}^m$ be a semialgebraic set such that the Zariski closure $\overline{S}$ is irreducible. Then a rational mapping $f:S \rightarrow \mathbb{R}^n$ is generically finite-to-one if and only if $\dim(f(S)) = \dim(S)$.  In particular, if $\dim(S)>n$ then $f$ must be generically infinite-to-one.
\end{lem}

\begin{rem}
If the rational mapping in Lemma 2.5 is infinite-to-one, then it holds that the fiber is infinite for almost all $s \in S$.  This can be seen, in particular, by inspecting the proof of Lemma 2.5. 
\end{rem}

In our context, the rational mapping of interest is the parametrization map $\phi_{\Gl}$, which maps into the positive definite cone $\textit{PD}(d)$.
We observe that a latent-factor graph $\Gl$ cannot be finite-to-one if the dimension of the domain $\Theta = \mathbb{R}^{D_V}_{\mathrm{reg}} \times \im(\tau)$ is larger than the dimension of $\textit{PD}(d)$. This gives a basic necessary condition.
\begin{figure}[t]
    \centering
    \tikzset{
      every node/.style={circle, inner sep=0.3mm, minimum size=0.45cm, draw, thick, black, fill=white, text=black},
      every path/.style={thick}
    }
    \begin{tikzpicture}[align=center]
      \node[fill=lightgray] (h1) at (0,1) {$h_{1}$};
      \node[] (1) at (-2,0) {1};
      \node[] (2) at (-1,0) {2};
      \node[] (3) at (-0,0) {3};
      \node[] (4) at (1,0) {4};
      \node[] (5) at (2,0) {5};
      
      \draw[red, dashed] [-latex] (h1) edge (1);
      \draw[red, dashed] [-latex] (h1) edge (2);
      \draw[red, dashed] [-latex] (h1) edge (3);
      \draw[red, dashed] [-latex] (h1) edge (4);
      \draw[red, dashed] [-latex] (h1) edge (5);
      
      \draw[blue] [-latex] (1) edge (2);
      \draw[blue] [-latex] (2) edge (3);
      \draw[blue] [-latex] (3) edge (4);
      \draw[blue] [-latex] (4) edge (5);
      \draw[blue] [-latex, bend right] (1) edge (3);
      \draw[blue] [-latex, bend right] (3) edge (5);
      
    \end{tikzpicture}
    \caption{Latent-factor graph that is (trivially) generically infinite-to-one.}
    \label{fig:inf-to-one}
\end{figure}
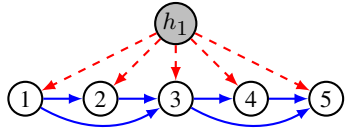

\begin{cor} \label{cor:infinite-to-one}
    A latent-factor graph $\Gl$ is generically infinite-to-one if $|D_V|+\dim(\mathrm{Im}(\tau)) > \binom{d+1}{2}$.
\end{cor}
\begin{proof}
    To apply Lemma \ref{lem:finite-to-one} we have to show that $\Theta = \mathbb{R}^{D_V}_{\mathrm{reg}} \times \im(\tau)$ is semialgebraic, its closure is irreducible and that the parametrization map $\phi_{\Gl}$ is rational. The first two claims are true since $\Theta$ is the polynomial image of an open semialgebraic set. Moreover, the map $\phi_{\Gl}$ is rational due to Cramer's rule. 
    
    Now, we study the dimensions of $\Theta$ and the image $\phi_{\Gl}(\Theta)$. The dimension of $\Theta$ is equal to $|D_V| + \dim(\im(\tau))$ since the dimension of the product of two semialgebraic sets is the sum of their individual dimensions \citep[Prop. 2.8.5]{bochnak1998real}. Since the image of $\phi_{\Gl}$ lies in the positive definite cone $\textit{PD}(d)$, we have 
    $$
        \dim(\phi_{\Gl}(\Theta)) \leq \dim(\textit{PD}(d)) = \binom{d+1}{2}.
    $$
    Thus, if $|D_V|+\dim(\mathrm{Im}(\tau)) > \binom{d+1}{2}$, then $\dim(\Theta) > \dim(\phi_{\Gl}(\Theta))$ and by Lemma \ref{lem:finite-to-one} we conclude that $\phi_{\Gl}$ is generically infinite-to-one.
\end{proof}

\begin{exmp}
Consider the graph in Figure \ref{fig:inf-to-one} where the latent structure is that of a one-factor model. By Theorem 2 in \citet{drton2007algebraic} we have $\dim(\im(\tau)) = 10$; with only one factor the dimension is equal to the number of edges from the latent node to the observed nodes, $|D_{LV}|=5$, plus the $5$ parameters appearing on the diagonal of the matrix $\Od$. But since the number of observed edges $|D_V|=6$ we have that $16 = |D_V|+\dim(\im(\tau)) > \binom{6}{2}=15$ and therefore the graph is generically infinite-to-one by Corollary \ref{cor:infinite-to-one}.
\end{exmp}

If a latent-factor graph is not trivially infinite-to-one by dimension comparison, then it becomes more difficult to decide whether it is generically infinite-to-one, generically finite-to-one or rationally identifiable. Figure \ref{fig:examples} shows latent-factor graphs that only have subtle differences in their structures but each of them has a different status of identifiability.

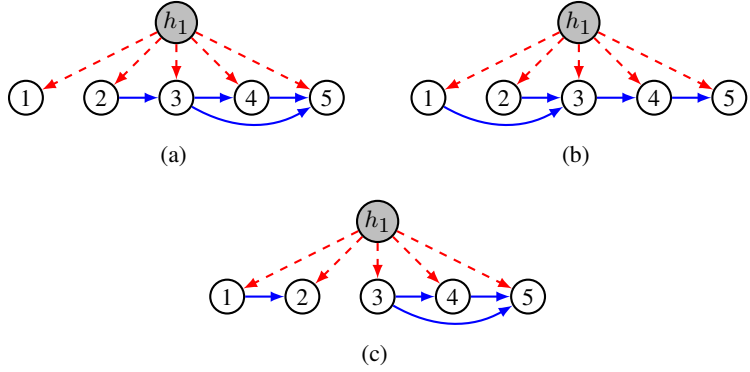
\begin{figure}[t]
    \subfloat[]{
    \centering
    \tikzset{
      every node/.style={circle, inner sep=0.3mm, minimum size=0.45cm, draw, thick, black, fill=white, text=black},
      every path/.style={thick}
    }
    \begin{tikzpicture}[align=center]
      \node[fill=lightgray] (h1) at (0,1) {$h_{1}$};
      \node[] (1) at (-2,0) {1};
      \node[] (2) at (-1,0) {2};
      \node[] (3) at (-0,0) {3};
      \node[] (4) at (1,0) {4};
      \node[] (5) at (2,0) {5};
      
      \draw[red, dashed] [-latex] (h1) edge (1);
      \draw[red, dashed] [-latex] (h1) edge (2);
      \draw[red, dashed] [-latex] (h1) edge (3);
      \draw[red, dashed] [-latex] (h1) edge (4);
      \draw[red, dashed] [-latex] (h1) edge (5);
      
      \draw[blue] [-latex] (2) edge (3);
      \draw[blue] [-latex] (3) edge (4);
      \draw[blue] [-latex] (4) edge (5);
      \draw[blue] [-latex, bend right] (3) edge (5);
    \end{tikzpicture}
    } \qquad
    \subfloat[]{
    \centering
    \tikzset{
      every node/.style={circle, inner sep=0.3mm, minimum size=0.45cm, draw, thick, black, fill=white, text=black},
      every path/.style={thick}
    }
    \begin{tikzpicture}[align=center]
      \node[fill=lightgray] (h1) at (0,1) {$h_{1}$};
      \node[] (1) at (-2,0) {1};
      \node[] (2) at (-1,0) {2};
      \node[] (3) at (-0,0) {3};
      \node[] (4) at (1,0) {4};
      \node[] (5) at (2,0) {5};
      
      \draw[red, dashed] [-latex] (h1) edge (1);
      \draw[red, dashed] [-latex] (h1) edge (2);
      \draw[red, dashed] [-latex] (h1) edge (3);
      \draw[red, dashed] [-latex] (h1) edge (4);
      \draw[red, dashed] [-latex] (h1) edge (5);
      
      \draw[blue] [-latex, bend right] (1) edge (3);
      \draw[blue] [-latex] (2) edge (3);
      \draw[blue] [-latex] (3) edge (4);
      \draw[blue] [-latex] (4) edge (5);
      
    \end{tikzpicture}
    }

    \subfloat[]{
    \centering
    \tikzset{
      every node/.style={circle, inner sep=0.3mm, minimum size=0.45cm, draw, thick, black, fill=white, text=black},
      every path/.style={thick}
    }
    \begin{tikzpicture}[align=center]
      \node[fill=lightgray] (h1) at (0,1) {$h_{1}$};
      \node[] (1) at (-2,0) {1};
      \node[] (2) at (-1,0) {2};
      \node[] (3) at (-0,0) {3};
      \node[] (4) at (1,0) {4};
      \node[] (5) at (2,0) {5};
      
      \draw[red, dashed] [-latex] (h1) edge (1);
      \draw[red, dashed] [-latex] (h1) edge (2);
      \draw[red, dashed] [-latex] (h1) edge (3);
      \draw[red, dashed] [-latex] (h1) edge (4);
      \draw[red, dashed] [-latex] (h1) edge (5);
      
      \draw[blue] [-latex] (1) edge (2);
      \draw[blue] [-latex] (3) edge (4);
      \draw[blue] [-latex] (4) edge (5);
      \draw[blue] [-latex, bend right] (3) edge (5);
      
    \end{tikzpicture}
    }
    
    \caption{Latent-factor graphs with one latent-factor. (a) Rationally identifiable.  (b) Generically finite-to-one but not rationally identifiable. (c) Generically infinite-to-one.}
    \label{fig:examples}
\end{figure}

\section{Main Identifiability Result} \label{sec:main-result}
The main idea underlying our sufficient condition for rational identifiability is to exploit the low rank structure of the latent covariance matrix
$$
    \Omega = \Od + \sum_{h \in \cL} \Gamma_h^{\top} \Gamma_h.
$$
Recall that $\Od \in \mathrm{diag}^{+}_d$ is diagonal and  $\Gamma_h$ is the $h$-th row of $\Gamma \in \mathbb{R}^{D_{\cL V}}$. For a node $v \in V$, denote by $\pa_V(v) = \{w \in V : w \rightarrow v \in D_V\}$ the set of \emph{observed parents} and by $\pa_{\cL}(v) = \{w \in \cL : w \rightarrow v \in D_{\cL V}\}$ the set of \emph{latent parents}.  So, $\pa(v) = \pa_V(v) \cup \pa_{\cL}(v)$. Focusing on a fixed node $v \in V$, it is our goal to find linear equations that determine the direct causal effects corresponding to the observed parents, that is, we aim to determine the vector  $\Lambda_{\pa_V(v), v}$. Our approach is to find suitable sets of observed nodes $Y,Z \subseteq V\setminus\{v\}$ and a set of latent nodes $H \subseteq \cL$  with $|H|=|Z|$ such that the latent covariance matrix contains a submatrix that satisfies
\begin{equation} \label{eq:low-rank matrix}
    \Omega_{Y, Z \cup \{v\}} = \sum_{h \in H} (\Gamma_h^{\top} \Gamma_h)_{Y, Z \cup \{v\}}
\end{equation}
and fails to have full column rank.  The drop in rank means that the entries of the submatrix exhibit algebraic relations, which we may then use to identify the targeted direct causal effects.

The equality in \eqref{eq:low-rank matrix} holds if (i) $Y \cap (Z \cup \{v\})=\emptyset$ and (ii) $\pa_{\cL}(Y) \cap \pa_{\cL}(Z \cup \{v\}) \subseteq H$.  Indeed, (i) ensures that $(\Od)_{Y, Z \cup \{v\}} = 0$ because the considered submatrix does not involve any diagonal elements.  And by (ii), the set $H$ contains all latent factors that have an effect on a node in $Y$ and at the same time an effect on a node in $Z \cup \{v\}$.  Assume there exists a triple of sets $(Y,Z,H)$ with $|H|=|Z|$ and satisfying (i) and (ii) above. Then
$$
    \rank\left(\Omega_{Y, Z \cup \{v\}}\right) = \rank\left(\sum_{h \in H} (\Gamma_h^{\top} \Gamma_h)_{Y, Z \cup \{v\}}\right) \leq |H| = |Z|,
$$
since the matrix $\Gamma_h^{\top} \Gamma_h$ has rank one for each $h \in \cL$. Hence the matrix $\Omega_{Y, Z \cup \{v\}}$ does not have full column rank. 
Moreover, suppose that we are able to ensure that the smaller submatrix $\Omega_{Y,Z}$ is of full column rank $|Z|$. Then, since the column ranks of $\Omega_{Y, Z \cup \{v\}}$ and $\Omega_{Y,Z}$ are equal, the vector $\Omega_{Y,v}$ must be a linear combination of the columns of $\Omega_{Y,Z}$, i.e., there exists $\psi  \in \mathbb{R}^{|Z|}$ such that $\Omega_{Y,Z} \cdot \psi = \Omega_{Y,v}$. 
Using the identity $(I_d - \Lambda)^{\top}\Sigma(I_d-\Lambda) = \Omega$ from \eqref{eq:Omega-start}, this is equivalent to
$$
    [(I_d-\Lambda)^{\top} \Sigma (I_d - \Lambda)]_{Y,v} - [(I_d-\Lambda)^{\top} \Sigma (I_d - \Lambda)]_{Y,Z} \cdot \psi = 0.
$$
Rewriting the matrix on the left we get the system of equations
\begin{equation} \label{eq:matrix-equation}
    \begin{pmatrix} [(I_d-\Lambda)^{\top} \Sigma]_{Y, \pa_V(v)} & [(I_d-\Lambda)^{\top} \Sigma (I_d - \Lambda)]_{Y,Z} \end{pmatrix}\cdot \begin{pmatrix} \Lambda_{\pa_V(v),v} \\ \psi\end{pmatrix} = [(I_d - \Lambda)^{\top} \Sigma]_{Y,v}.
\end{equation}
Now, if we make sure the matrix on the left-hand side in \eqref{eq:matrix-equation} is square and invertible, we can solve the system for the unknown parameters $\Lambda_{\pa_V(v), v}$.  However, for this to be useful for parameter identification, suitable entries of $\Lambda$ must already be known from earlier similar calculations in order to determine the coefficient matrix and the vector on the right-hand side of \eqref{eq:matrix-equation}.

\begin{exmp}
    Consider the graph in Figure \ref{fig:examples} (a). Since there is one latent factor having dense effect on all observed variables, the parameter matrix $\Gamma$ is given by the row vector $(\gamma_{11}, \ldots, \gamma_{15})$. 
    Now focus on node $v=3$ which only has a single observed parent. We aim to recover the effect $\Lambda_{\pa_V(3),3} = \lambda_{23}$ and we claim that the triple $(Y,Z,H)=(\{2,4\},\{1\},\{h_1\})$ satisfies the properties discussed above. 
    Clearly, $|H|=|Z|$, we have empty intersection $Y \cap (Z \cup \{v\})$ and the only common latent parent of $Y$ and $Z \cup \{v\}$ is $h_1$, i.e., $\pa_{\cL}(Y) \cap \pa_{\cL}(Z \cup \{v\}) \subseteq H$. By inspecting the rank one submatrix 
    \[
        \Omega_{Y, Z \cup \{v\}} 
        =
        \begin{pmatrix}
            \gamma_{12} \\
            \gamma_{14}
        \end{pmatrix} 
        \cdot
        \begin{pmatrix}
            \gamma_{11} & \gamma_{13}
        \end{pmatrix}
        = 
        \begin{pmatrix}
            \gamma_{11} \gamma_{12} \ & \ \gamma_{12} \gamma_{13} \\
            \gamma_{11} \gamma_{14} \ & \ \gamma_{13} \gamma_{14}
        \end{pmatrix} 
    \]
    we can easily deduce the relation
    \[
        \Omega_{Y, Z}  \cdot \frac{\gamma_{13}}{\gamma_{11}} = \Omega_{Y,v}
    \]
    which holds true for generic choices of $\gamma_{11}$, i.e., for $\gamma_{11} \neq 0$. In other words, the parameter $\psi$ is equal to $\gamma_{13}/\gamma_{11}$
    and the equation system \eqref{eq:matrix-equation} is given by
    \[
        \begin{pmatrix} 
        \sigma_{22} & \sigma_{12} \\
        -\lambda_{34} \sigma_{23} + \sigma_{24} \ & \ -\lambda_{34} \sigma_{13} + \sigma_{14} 
        \end{pmatrix} 
        \begin{pmatrix} 
        \lambda_{23} \\
        \psi
        \end{pmatrix} = 
        \begin{pmatrix} 
        \sigma_{23} \\
        -\lambda_{34} \sigma_{33} + \sigma_{34}
        \end{pmatrix}
    \]
    where $\sigma_{ij}$ is the $ij$-th entry of the covariance matrix $\Sigma$. If we already knew that the effect $\lambda_{34}$ is given by a rational function in $\Sigma$, then we could also recover the effect $\lambda_{23}$ by a rational function of $\Sigma$ since the matrix on the left-hand side is quadratic and generically invertible. 
\end{exmp}

Our main result shows that the above story can be made practical and yields a criterion to recursively identify columns in $\Lambda$.  Importantly, the imposed conditions can all be translated into combinatorial conditions on the considered latent-factor graph.  The resulting method is proven correct in Theorem~\ref{thm:lfhtc} below.  Before stating the theorem we define the necessary graphical concepts, which involve special types of paths that we term latent-factor half-treks. Recall that a path from node $v$  to $w$ in a latent-factor graph $\Gl = (V \cup \cL, D)$ is a sequence of edges that connects the consecutive nodes in a sequence of nodes beginning in $v$ and ending in $w$. 

\begin{defn}[Latent-factor half-trek]\label{def:latent_halftrek}
  A path $\pi$ in the latent-factor graph $\Gl$ is a
  \emph{latent-factor half-trek from source $v$ to target $w$} if it is a path
  from $v \in V$ to $w \in V$ in $\Gl$ and is of the form
  \[
    v \rightarrow x_1 \rightarrow \dots \rightarrow x_n \rightarrow w
  \]
  or of the form
  \[
    v \leftarrow h \rightarrow x_1 \rightarrow \dots \rightarrow x_n \rightarrow w
  \]
  for $x_1,\dots,x_n\in V$ and for some $h\in \cL$.
\end{defn}

The name latent-factor half-trek is inspired by the customary notion of a trek, which is a pair of directed paths $(\pi_1, \pi_2)$ that share the same source node.
If a latent-factor half-trek is of the first form in Definition \ref{def:latent_halftrek}, we say that the left-hand side of $\pi$, written $\Left(\pi)$, is the node $v$ and the right-hand side, written $\Right(\pi)$, is the set of nodes $\{v, x_1, \ldots, x_n, w\}$. In the second case $\Left(\pi) = \{v, h\}$ and $\Right(\pi) = \{h, x_1, \ldots, x_n, w\}$. A latent-factor half-trek from $v$ to $v$ may have no edges, in this case $\Left(\pi) = \Right(\pi)=\{v\}$ and the half-trek is called trivial. For a set of $n$ latent-factor half-treks, $\Pi = \{\pi_1, \ldots, \pi_n\}$, let $v_i$ and $w_i$ be the source and the target of $\pi_i$. If the sources are all distinct and the targets are all distinct, then we say that $\Pi$ is a system of latent-factor half-treks from $A=\{v_1, \ldots, v_n\}$ to $B=\{w_1, \ldots, w_n\}$. A set of latent-factor half-treks $\Pi = \{\pi_1, \ldots, \pi_n\}$ has no sided intersection if
$$
\Left(\pi_i) \cap \Left(\pi_j) = \emptyset = \Right(\pi_i) \cap \Right(\pi_j) \ \mathrm{for} \ \mathrm{all} \ i \neq j.
$$

 \begin{exmp}
     Consider the graph in Figure \ref{fig:examples} (a). Then the system of latent-factor half-treks
     $$
        \{\pi_1: 5 \leftarrow h_1 \rightarrow 3, \quad \pi_2: 4 \rightarrow 5\}
     $$
     has no sided intersection. On the other hand, the system 
     $$
        \{\widetilde{\pi}_1: 2 \leftarrow h_1 \rightarrow 3, \quad \widetilde{\pi}_2: 3 \rightarrow 4 \rightarrow 5\}
     $$
     has sided intersection since $\Right(\widetilde{\pi}_1) \cap \Right(\widetilde{\pi}_2) = \{3\}$.
 \end{exmp}

\begin{defn}[Latent-factor half-trek reachability]\label{def:latent_htr}
  Let $v,w\in V$ be two distinct observed nodes in a latent-factor graph $\Gl$.  Let $H\subseteq \cL$ be a set of latent factors.  If there exists a latent-factor half-trek from $v$ to
  $w$ through the latent-factor graph $\Gl$, which does not pass through any node in $H$, then we say that \emph{$w$ is half-trek
  reachable from $v$ while avoiding $H$}, and write
  $w\in\htr_H(v)$. For a set $U\subseteq V$, we write $w\in\htr_H(U)$ if
  $w\in\htr_H(u)$ for some $u\in U$.
\end{defn}

 \begin{exmp}
     Consider the graph in Figure \ref{fig:examples} (a), and let $H=\emptyset$. Then $2 \in \htr_H(1)$ since there is the latent-factor half-trek $1 \leftarrow h_1 \rightarrow 2$ and $h_1 \not\in H$. But if $H=\{h_1\}$, then $\htr_H(1) = \emptyset$ since there is no latent-factor half-trek from node $1$ to  any other node in the graph while avoiding the node $h_1$.
 \end{exmp}

\begin{defn}[Latent-factor half-trek criterion]\label{def:latent_htc}
  Given a node $v\in V$, the triple
  $(Y,Z,H)\in 2^{V\setminus \{v\}}\times 2^{V\setminus \{v\}} \times 2^{\cL}$ satisfies the
  \emph{latent-factor half-trek criterion} (LF-HTC) with respect to
  $v$ if
  \begin{enumerate}[(i)]
  \item $|Y| = |\pa_V(v)| + |H|$ and $|Z|=|H|$ with $Z \cap \pa_V(v) = \emptyset$,
  \item $Y\cap (Z\cup \{v\})=\emptyset$ and $\pa_{\cL}(Y) \cap \pa_{\cL}(Z \cup \{v\}) \subseteq H$, and
  \item there exists a system of latent-factor half-treks with no sided intersection from $Y$ to $Z\cup \pa_V(v)$ in $\Gl$, such that for each $z\in Z$, the half-trek terminating at $z$ takes the form $y\leftarrow h \rightarrow z$ for some $y\in Y$ and some $h\in H$.
  \end{enumerate}
\end{defn}

If a triple $(Y,Z,H)$ satisfies the LF-HTC with respect to a node $v$, then
condition (ii) ensures that the submatrix $\Omega_{Y, Z \cup \{v\}}$ of the latent covariance matrix can be written as in \eqref{eq:low-rank matrix} and, since $|Z|=|H|$, the submatrix does not have full column rank. Moreover, condition (iii) ensures that the matrix on the left-hand side of \eqref{eq:matrix-equation} is invertible. The latter claim will be established by means of an application of the Gessel-Viennot-Lindstr\"om Lemma \citep{gessel1985binomial, lindstrom1973on}. We now state our main result; its proof is deferred to Section \ref{sec:proof-lfhtc}. For a directed edge $u \rightarrow y \in D$ we say that $y$ is the \textit{head} of the edge.

\begin{thm}[LF-HTC-identifiability] \label{thm:lfhtc}
  Suppose that the triple
  $(Y,Z,H)\in 2^{V\setminus \{v\}}\times 2^{V\setminus \{v\}}\times
  2^\cL$ satisfies the LF-HTC with respect to $v\in V$. If all
  directed edges $u\to y\in D_V$ with head
  $y\in Z\cup(Y\cap \htr_H(Z\cup\{v\}))$ are rationally
  identifiable, then all directed edges in $D_V$ with $v$ as a head are
  rationally identifiable.
\end{thm}

This theorem yields the basis for an efficient algorithm that recursively solves for all direct causal effects corresponding to the edges $D_V$ in a latent-factor graph. That is, we recover the matrix $\Lambda$ column-by-column. The corresponding algorithm is detailed in Section \ref{sec:computation}. We refer to a latent-factor graph $\Gl$ as \emph{LF-HTC-identifiable} if all columns of  $\Lambda$ may be recovered recursively by Theorem \ref{thm:lfhtc}.

\begin{exmp} \label{ex:first-example}
    The latent-factor graph in Figure \ref{fig:examples} (a) is LF-HTC-identifiable. To see this, we recursively check all nodes $v \in V=\{1,2,3,4,5\}$. That is, for each $v \in V$ we find a triple $(Y,Z,H)$ that satisfies the LF-HTC such that all nodes in $Z\cup(Y\cap \htr_H(Z\cup\{v\}))$ were already checked successfully to satisfy the LF-HTC in the steps before.\\
    $\underline{v=1,2}$: The triple $(Y,Z,H) = (\emptyset, \emptyset, \emptyset)$ trivially satisfies the LF-HTC since $\pa_V(v) = \emptyset$. \\
    $\underline{v=4}$: Let $(Y,Z,H) = (\{2,3\}, \{1\}, \{h_1\})$. Conditions (i) and (ii) are easily checked and for condition (iii) consider the system of latent-factor half-treks $\{3, 2 \leftarrow h_1 \rightarrow 1\}$ where $3$ corresponds to the trivial trek from $3$ to $3$. Finally, note that $Y \cap \htr_H(Z\cup\{v\}) = \{2,3\} \cap \{4,5\} = \emptyset$ and that the node $1 \in Z$ was already checked successfully in the last step. \\
    $\underline{v=3}$: Let $(Y,Z,H) = (\{2,4\}, \{1\}, \{h_1\})$. Then the system of latent-factor half-treks $\{2, 4 \leftarrow h_1 \rightarrow 1\}$ satisfies (iii) and $Z \cup (Y \cap \htr_H(Z\cup\{v\})) = \{1,4\}$. \\
    $\underline{v=5}$: Let $(Y,Z,H) = (\{2,3,4\}, \{1\}, \{h_1\})$. Then the system of latent-factor half-treks $\{3, 4, 2 \leftarrow h_1 \rightarrow 1\}$ satisfies (iii) and $Z \cup (Y \cap \htr_H(Z\cup\{v\})) = \{1\}$. 
\end{exmp}

If the observed part $(V, D_V)$ of a latent-factor graph does not contain directed cycles, then the latent-factor graph is said to be \emph{acyclic}. Moreover, we say that a latent-factor graph is \emph{bow-free} if it does not contain any two observed vertices $v,w \in V$ such that there is a directed edge between $v$ and $w$ and, in addition, there is a latent factor $h \in \cL$ that has directed edges pointing to both $v$ and $w$. As a special case of Theorem \ref{thm:lfhtc} we have the following straightforward observation.

\begin{cor}
Bow-free acyclic latent-factor graphs are rationally identifiable.
\end{cor}
\begin{proof}
Let  $\Gl=(V \cup \cL, D)$ be a latent-factor graph. It is easy to see that for every node $v \in V$ the triple $(Y,Z,H)=(\pa_V(v), \emptyset, \emptyset)$ satisfies the LF-HTC with respect to $v$ since $v$ and $\pa_V(v)$ do not have a common latent parent (i.e., $\pa_{\cL}(\pa_{V}(v)) \cap \pa_{\cL}(v) = \emptyset$). The observed part $(V,D_V)$ is a directed aycylic graph (DAG) and therefore induces at least one topological ordering  $\prec$ on $V$, that is, an ordering such that $v \rightarrow w \in D_V$ only if $v \prec w$. Importantly, all parents $w \in \pa_V(v)$ are predecessors of $v$ with respect to $\prec$. Thus by Theorem \ref{thm:lfhtc} we can determine rational identifiability of all edges in $D_V$ in a step-wise manner according to the ordering $\prec$ and using the triple $(\pa_V(v), \emptyset, \emptyset)$ for each $v \in V$. We conclude that $\Gl$ is LF-HTC-identifiable and hence, in particular, rationally identifiable.
\end{proof}

\section{Latent Projections} \label{sec:latent-proj}

As mentioned in the introduction, previous criteria for rational identifiability of direct causal effects operate on mixed graphs obtained by a projection. These projections can be defined for general directed graphs with hidden variables (\citeauthor{maathuis2019handbook}, \citeyear{maathuis2019handbook}, Chap.~2 and \citeauthor{pearl2009causality}, \citeyear{pearl2009causality}, Chap.~2), but we treat the special case of latent-factor graphs:

\begin{defn}[\citeauthor{maathuis2019handbook}, \citeyear{maathuis2019handbook}, Chap.~2] \label{def:latent-projection}
  Let $\Gl = (V\cup\cL, D)$ be a latent-factor graph. Define a new graph starting with the induced subgraph $G^{\prime} = (V, D_V)$ and add edges as follows:
  $$
    \mathrm{Whenever} \ v \leftarrow h \rightarrow w \ \mathrm{in} \ \Gl \ \mathrm{for} \ h \in \cL \ \mathrm{and} \  v,w \in V, \ \mathrm{add} \ v \leftrightarrow w \ \mathrm{to} \  G^{\prime}.
  $$
  The mixed graph $G^{\prime} = (V, D_V, B)$ is the \emph{latent projection} of $\Gl$, where $B$ is the collection of bidirected edges $v \leftrightarrow w$. They have no orientation, i.e., $v \leftrightarrow w \in B$ if and only if $w \leftrightarrow v \in B$.
\end{defn}


Every mixed graph defines a covariance model. Denote $\textit{PD}(B) \subseteq \textit{PD}(d)$ the subcone of matrices with support $B$, that is, for $\Omega = (\omega_{vw}) \in \textit{PD}(B)$ we have $\omega_{vw} = 0$ if $v \neq w$ and $v \leftrightarrow w \notin B$.

\begin{defn}
  The covariance model given by a mixed graph $G^{\prime} = (V, D_V, B)$  with $V=|d|$ is the family of covariance matrices
  $$
    \Sigma = (I_d - \Lambda)^{-\top}\Omega(I_d-\Lambda)^{-1}
  $$
  obtained from choices of $\Lambda \in \mathbb{R}^{D_V}_{\mathrm{reg}}$ and $\Omega \in \textit{PD}(B)$.
\end{defn}

For any latent-factor graph, the cone of latent covariance matrices $\im(\tau)$ is clearly a subset of  $\textit{PD}(B)$, the cone of latent covariance matrices of the latent projection. Thus, a covariance model given by a latent-factor graph is a submodel of the covariance model given by its latent projection.  More details on the at times subtle differences between  $\im(\tau)$ and $\textit{PD}(B)$ can be found in \cite{drtonyu2010}.

In the remainder of this section, we focus on the predecessor of the LF-HTC that operates on mixed graphs, namely the original half-trek criterion (HTC) of \citet{foygel2012halftrek}. We say that a mixed graph is HTC-identifiable if it is rationally identifiable by this criterion.

At first sight, it appears as if the HTC coincides with the version of the LF-HTC obtained by only allowing $H=Z=\emptyset$; compare Def.~4 in \citet{foygel2012halftrek} with Definition~\ref{def:latent_htc} here.  However, as we will show below there is a subtle difference in the way systems of half-treks with no sided intersection are defined.  Indeed, in the setting of the LF-HTC two half-treks may also intersect at latent nodes, whereas in the HTC intersections are only possible at observed nodes.  Intuitively, each bidirected edge in a latent projection can amount to confounding induced by a separate latent variable.  Before highlighting this subtlety, we first exemplify an application of HTC.

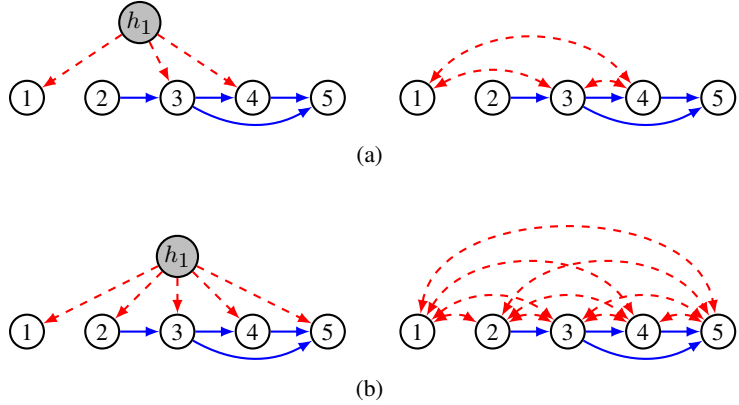
\begin{figure}[t]
    \subfloat[]{
    \centering
    \tikzset{
      every node/.style={circle, inner sep=0.3mm, minimum size=0.45cm, draw, thick, black, fill=white, text=black},
      every path/.style={thick}
    }
    \begin{tikzpicture}[align=center]
      \node[fill=lightgray] (h1) at (-0.5,1) {$h_{1}$};
      \node[] (1) at (-2,0) {1};
      \node[] (2) at (-1,0) {2};
      \node[] (3) at (-0,0) {3};
      \node[] (4) at (1,0) {4};
      \node[] (5) at (2,0) {5};
      
      \draw[red, dashed] [-latex] (h1) edge (1);
      \draw[red, dashed] [-latex] (h1) edge (3);
      \draw[red, dashed] [-latex] (h1) edge (4);
      
      \draw[blue] [-latex] (2) edge (3);
      \draw[blue] [-latex] (3) edge (4);
      \draw[blue] [-latex] (4) edge (5);
      \draw[blue] [-latex, bend right] (3) edge (5);
    \end{tikzpicture}
    \qquad
    \begin{tikzpicture}[align=center]
      \node[] (1) at (-2,0) {1};
      \node[] (2) at (-1,0) {2};
      \node[] (3) at (-0,0) {3};
      \node[] (4) at (1,0) {4};
      \node[] (5) at (2,0) {5};
      
      \draw[red, dashed] [latex-latex, bend left] (1) edge (3);
      \draw[red, dashed] [latex-latex, bend left] (3) edge (4);
      \draw[red, dashed] [latex-latex, bend left, out=50, in=130] (1) edge (4);
      
      \draw[blue] [-latex] (2) edge (3);
      \draw[blue] [-latex] (3) edge (4);
      \draw[blue] [-latex] (4) edge (5);
      \draw[blue] [-latex, bend right] (3) edge (5);
    \end{tikzpicture}
    }

    \subfloat[]{
    \centering
    \tikzset{
      every node/.style={circle, inner sep=0.3mm, minimum size=0.45cm, draw, thick, black, fill=white, text=black},
      every path/.style={thick}
    }
    \begin{tikzpicture}[align=center]
      \node[fill=lightgray] (h1) at (0,1) {$h_{1}$};
      \node[] (1) at (-2,0) {1};
      \node[] (2) at (-1,0) {2};
      \node[] (3) at (-0,0) {3};
      \node[] (4) at (1,0) {4};
      \node[] (5) at (2,0) {5};
      
      \draw[red, dashed] [-latex] (h1) edge (1);
      \draw[red, dashed] [-latex] (h1) edge (2);
      \draw[red, dashed] [-latex] (h1) edge (3);
      \draw[red, dashed] [-latex] (h1) edge (4);
      \draw[red, dashed] [-latex] (h1) edge (5);
      
      \draw[blue] [-latex] (2) edge (3);
      \draw[blue] [-latex] (3) edge (4);
      \draw[blue] [-latex] (4) edge (5);
      \draw[blue] [-latex, bend right] (3) edge (5);
    \end{tikzpicture}
    \qquad
    \begin{tikzpicture}[align=center]
      \node[] (1) at (-2,0) {1};
      \node[] (2) at (-1,0) {2};
      \node[] (3) at (-0,0) {3};
      \node[] (4) at (1,0) {4};
      \node[] (5) at (2,0) {5};
      
      \draw[red, dashed] [latex-latex, bend left] (1) edge (2);
      \draw[red, dashed] [latex-latex, bend left] (2) edge (3);
      \draw[red, dashed] [latex-latex, bend left] (3) edge (4);
      \draw[red, dashed] [latex-latex, bend left] (4) edge (5);
      \draw[red, dashed] [latex-latex, bend left, out=40, in=140] (1) edge (3);
      \draw[red, dashed] [latex-latex, bend left, out=40, in=140] (2) edge (4);
      \draw[red, dashed] [latex-latex, bend left, out=40, in=140] (3) edge (5);
      \draw[red, dashed] [latex-latex, bend left, out=60, in=120] (1) edge (4);
      \draw[red, dashed] [latex-latex, bend left, out=60, in=120] (2) edge (5);
      \draw[red, dashed] [latex-latex, bend left, out=80, in=100] (1) edge (5);
      
      \draw[blue] [-latex] (2) edge (3);
      \draw[blue] [-latex] (3) edge (4);
      \draw[blue] [-latex] (4) edge (5);
      \draw[blue] [-latex, bend right] (3) edge (5);
    \end{tikzpicture}
    }
    \caption{Latent-factor graphs and their latent projection.}
    \label{fig:example-projection}
    
\end{figure}

\begin{exmp} \label{ex:latent-projection}
    Figure \ref{fig:example-projection} shows two latent-factor graphs and their latent projection. Both latent-factor graphs are LF-HTC-identifiable, cf.~Example \ref{ex:first-example}. But only the latent projection in the upper panel (a) is HTC-identifiable while the latent projection in panel (b) is generically infinite-to-one. The latter is easily seen since the number of model parameters corresponding to the mixed graph is larger than the dimension $\binom{d+1}{2}$ of the space $\textit{PD}(d)$, see e.g.~Proposition 2 in \citet{foygel2012halftrek}. 
\end{exmp}

Comparing the graphs in Figure \ref{fig:example-projection}, the latent-factor graphs on the left-hand side  assume that all unobserved confounding is caused by a single latent factor. In contrast, for the latent projections on the right-hand side, there may be multiple latent factors that are the sources of confounding represented by bidirected edges. This leads to rational identifiability of the latent-factor graphs while the projection on the mixed graphs may be generically infinite-to-one. 


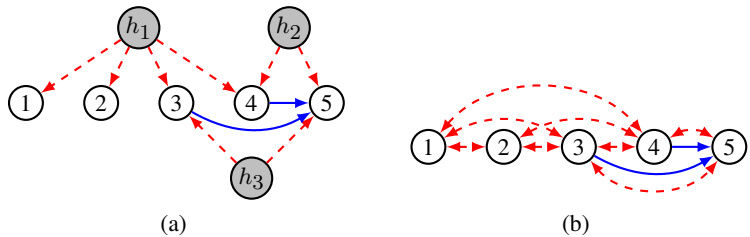
\begin{figure}[b]
    \subfloat[]{
    \centering
    \tikzset{
      every node/.style={circle, inner sep=0.3mm, minimum size=0.45cm, draw, thick, black, fill=white, text=black},
      every path/.style={thick}
    }
    \begin{tikzpicture}[align=center]
      \node[fill=lightgray] (h1) at (-0.5,1) {$h_{1}$};
      \node[fill=lightgray] (h2) at (1.5,1) {$h_{2}$};
      \node[fill=lightgray] (h3) at (1,-1) {$h_{3}$};
      \node[] (1) at (-2,0) {1};
      \node[] (2) at (-1,0) {2};
      \node[] (3) at (-0,0) {3};
      \node[] (4) at (1,0) {4};
      \node[] (5) at (2,0) {5};
      
      \draw[red, dashed] [-latex] (h1) edge (1);
      \draw[red, dashed] [-latex] (h1) edge (2);
      \draw[red, dashed] [-latex] (h1) edge (3);
      \draw[red, dashed] [-latex] (h1) edge (4);
      \draw[red, dashed] [-latex] (h2) edge (4);
      \draw[red, dashed] [-latex] (h2) edge (5);
      \draw[red, dashed] [-latex] (h3) edge (3);
      \draw[red, dashed] [-latex] (h3) edge (5);
      
      \draw[blue] [-latex] (4) edge (5);
      \draw[blue] [-latex, bend right] (3) edge (5);
    \end{tikzpicture}
    }
    \qquad
    \subfloat[]{
    \centering
    \tikzset{
      every node/.style={circle, inner sep=0.3mm, minimum size=0.45cm, draw, thick, black, fill=white, text=black},
      every path/.style={thick}
    }
    \begin{tikzpicture}[align=center]
      \node[] (1) at (-2,0) {1};
      \node[] (2) at (-1,0) {2};
      \node[] (3) at (-0,0) {3};
      \node[] (4) at (1,0) {4};
      \node[] (5) at (2,0) {5};
      
      \draw[red, dashed] [latex-latex] (1) edge (2);
      \draw[red, dashed] [latex-latex] (2) edge (3);
      \draw[red, dashed] [latex-latex] (3) edge (4);
      \draw[red, dashed] [latex-latex, bend left] (1) edge (3);
      \draw[red, dashed] [latex-latex, bend left] (2) edge (4);
      \draw[red, dashed] [latex-latex, bend left, out=50, in=130] (1) edge (4);
      \draw[red, dashed] [latex-latex, bend left] (4) edge (5);
      \draw[red, dashed] [latex-latex, bend right, out=-50, in=-130] (3) edge (5);
      
      \draw[blue] [-latex] (4) edge (5);
      \draw[blue] [-latex, bend right] (3) edge (5);
    \end{tikzpicture}
    }
    \caption{Latent-factor graph that is generically infinite-to-one but its latent projection is HTC-identifiable.}
    \label{fig:counterexmp}
\end{figure}

Surprisingly, a mixed graph $G^{\prime}$ being rationally identifiable does \textit{not} imply that all latent-factor graphs $G^{\cL}$ having $G^{\prime}$ as their latent projection are rationally identifiable.
Recall that in the case of rational identifiability of the latent projection there may be a proper algebraic subset $A$ of the Zariski closure of $\mathbb{R}^{D_V}_{\mathrm{reg}} \times  \textit{PD}(B)$ such that identification is not possible on $A$. If the dimensionality of the cone of latent covariance matrices $\im(\tau)$ is strictly smaller than the dimension of $\textit{PD}(B)$, it can therefore happen that $\Theta = \mathbb{R}^{D_V}_{\mathrm{reg}}  \times \im(\tau) \subseteq A$ and the latent-factor graph is generically infinite-to-one. As an example, the latent projection in Figure \ref{fig:counterexmp} is HTC-identifiable while the latent-factor graph itself is generically infinite-to-one. In this example,  $\dim(\im(\tau)) = 11$ while $\dim(\textit{PD}(B))=13$.  Hence, although the model given by the graph to the left is still a submodel of the one given by the graph to the right, the relevant notion of genericity is different, referring to proper subsets of $\textit{PD}(B)$ and of $\im(\tau)$, respectively.

In the experiments in Section \ref{sec:experiments}, we systematically compare LF-HTC-identifiability of latent-factor graphs with HTC-identifiability applied to the corresponding latent projection.

\section{Computation} \label{sec:computation}

In this section we propose an efficient algorithm for deciding whether a latent-factor graph is LF-HTC-identifiable. It is similar to the algorithm of the original half-trek criterion in \citet{foygel2012halftrek} and makes use of maximum flows in a special flow graph $\Gf=(V_f, D_f)$ from a designated source node $s \subseteq V_f$ to a target node $t \subseteq V_f$. The standard maximum-flow framework is introduced in \citet{cormen2009introduction}. We highlight that the maximum flow can be computed in polynomial time and the complexity is $\mathcal{O}((|V_f|+r)^3)$ where $r \leq |D_f|/2$ is the number of reciprocal edge pairs in $D_f$. A reciprocal edge pair is a pair $v \rightarrow w$ and $w \rightarrow v$ for distinct nodes $v \neq w \in V_f$. 

Let $\Gl$ be a latent-factor graph, and fix a node $v \in V$. Then we denote by  \texttt{LF-HTC}$(\Gl,v)$ the decision problem whether there exists a triple $(Y,Z,H) \in 2^{V\setminus\{v\}}\times 2^{V\setminus \{v\}}\times 2^\cL$ satisfying the LF-HTC for $v \in V$ in $\Gl$.  To solve this problem, we first address a subproblem by assuming that we are given a fixed set $H \subseteq \cL$ and a fixed set $Z \subseteq \ch(H) \setminus (\{v\} \cup \pa_V(v))$ such that $|Z|=|H|$. 
Since the second part of condition (ii) of the LF-HTC is equivalent to $Y \cap \ch(\pa_{\cL}(Z \cup \{v\}) \setminus H) = \emptyset$, the set $A = V\setminus (Z \cup \{v\} \cup \ch(\pa_{\cL}(Z \cup \{v\}) \setminus H))$ is the set of ``allowed'' nodes that may contain a set $Y \subseteq A$ such that $(Y,Z,H)$ satisfies the LF-HTC with respect to $v$. We are able to prove the existence or inexistence of such a set $Y$ efficiently by one maximum flow computation on a suitable flow graph $\Gf(v, A, Z)=(V_f, D_f)$. 

The flow graph is defined as follows: Let $V^{\prime}$ and $\cL^{\prime}$ be copies of the sets $V$ and $\cL$. Then the graph contains the nodes $V_f = (A \cup \cL) \cup (V^{\prime} \cup \cL^{\prime}) \cup \{s,t\}$, where $s$ is a source node and $t$ is a sink node. The set of edges $D_f$ contains
\begin{itemize}
    \item[(a)] $s \rightarrow a$ for all $a \in A$,
    \item[(b)] $a \rightarrow w$ if $a \in A$ and $w \rightarrow a \in D_{\cL V}$,
    \item[(c)] $w \rightarrow w^{\prime}$ for all $w \in A \cup \cL$,
    \item[(d)] $u^{\prime} \rightarrow w^{\prime}$ for all $u \rightarrow w \in D_{\cL V}$ and for all $u \rightarrow w \in D_V $ such that $w \notin Z$,
    \item[(e)] $w^{\prime} \rightarrow t$ for all $w \in \pa_V(v) \cup Z$.
\end{itemize}
We assign to all edges capacity $\infty$. The source node $s$ and the target node $t$ have capacity $\infty$ while all other nodes have capacity 1. Note that, by construction, no flow in $\Gf(v, A, Z)$ can exceed $|\pa_V(v)| + |Z|$ in size, therefore one may replace the infinite capacities with  $|\pa_V(v)| + |Z|$ in practice. An example of a flow graph is shown in Figure \ref{fig:example-flow-graph} (b).

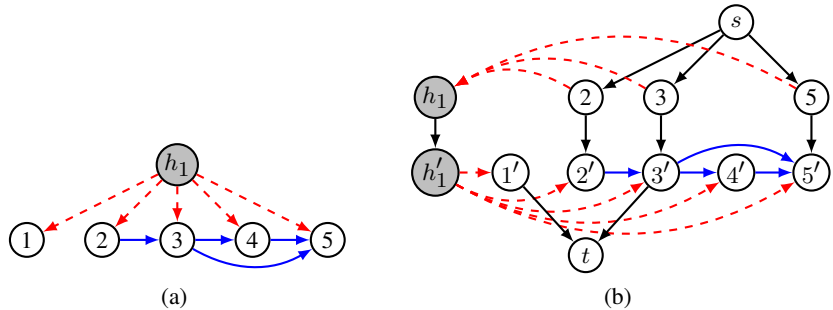
\begin{figure}[b]

    \subfloat[]{
    \centering
    \tikzset{
      every node/.style={circle, inner sep=0.3mm, minimum size=0.45cm, draw, thick, black, fill=white, text=black},
      every path/.style={thick}
    }
    \begin{tikzpicture}[align=center]
      \node[fill=lightgray] (h1) at (0,1) {$h_{1}$};
      \node[] (1) at (-2,0) {1};
      \node[] (2) at (-1,0) {2};
      \node[] (3) at (-0,0) {3};
      \node[] (4) at (1,0) {4};
      \node[] (5) at (2,0) {5};
      
      \draw[red, dashed] [-latex] (h1) edge (1);
      \draw[red, dashed] [-latex] (h1) edge (2);
      \draw[red, dashed] [-latex] (h1) edge (3);
      \draw[red, dashed] [-latex] (h1) edge (4);
      \draw[red, dashed] [-latex] (h1) edge (5);
      
      \draw[blue] [-latex] (2) edge (3);
      \draw[blue] [-latex] (3) edge (4);
      \draw[blue] [-latex] (4) edge (5);
      \draw[blue] [-latex, bend right] (3) edge (5);
    \end{tikzpicture}
    }
    \qquad
    \subfloat[]{
    \centering
    \tikzset{
      every node/.style={circle, inner sep=0.3mm, minimum size=0.45cm, draw, thick, black, fill=white, text=black},
      every path/.style={thick}
    }
    \begin{tikzpicture}[align=center]
      \node[fill=lightgray] (h1) at (-3,0) {$h_{1}$};
      \node[] (2) at (-1,0) {2};
      \node[] (3) at (-0,0) {3};
      \node[] (5) at (2,0) {5};
      
      \node[fill=lightgray] (h1p) at (-3,-1) {$h_{1}^{\prime}$};
      \node[] (1p) at (-2,-1) {$1^{\prime}$};
      \node[] (2p) at (-1,-1) {$2^{\prime}$};
      \node[] (3p) at (-0,-1) {$3^{\prime}$};
      \node[] (4p) at (1,-1) {$4^{\prime}$};
      \node[] (5p) at (2,-1) {$5^{\prime}$};
      
      \node[] (s) at (1,1) {$s$};
      \node[] (t) at (-1,-2.1) {$t$};
      
      \draw[] [-latex] (s) edge (2);
      \draw[] [-latex] (s) edge (3);
      \draw[] [-latex] (s) edge (5);
       
      \draw[red, dashed] [-latex, bend right] (2) edge (h1);
      \draw[red, dashed] [-latex, bend right] (3) edge (h1);
      \draw[red, dashed] [-latex, bend right] (5) edge (h1);
      
      \draw[] [-latex] (h1) edge (h1p);
      \draw[] [-latex] (2) edge (2p);
      \draw[] [-latex] (3) edge (3p);
      \draw[] [-latex] (5) edge (5p);
      
      \draw[red, dashed] [-latex] (h1p) edge (1p);
      \draw[red, dashed] [-latex, bend right] (h1p) edge (2p);
      \draw[red, dashed] [-latex, bend right] (h1p) edge (3p);
      \draw[red, dashed] [-latex, bend right] (h1p) edge (4p);
      \draw[red, dashed] [-latex, bend right] (h1p) edge (5p);
      
      \draw[blue] [-latex] (2p) edge (3p);
      \draw[blue] [-latex] (3p) edge (4p);
      \draw[blue] [-latex] (4p) edge (5p);
      \draw[blue] [-latex, bend left] (3p) edge (5p);
      
      \draw[] [-latex] (1p) edge (t);
      \draw[] [-latex] (3p) edge (t);
    \end{tikzpicture}
    }
    \caption{Using maximum-flow to find a set $Y \subseteq A$ such that the triple $(Y,Z,H)$ with fixed sets $H=\{h_1\}$ and $Z=\{1\}$ satisfies the LF-HTC with respect to $v=4$. The set of allowed nodes is $A=\{2,3,5\}$. (a) The concerned latent-factor graph. (b) The corresponding flow graph $\Gf(v, A, Z)$. }
    \label{fig:example-flow-graph}
\end{figure}

\newpage
Let $\texttt{MaxFlow}(\Gf(v, A, Z))$ be the maximum flow from $s$ to $t$ in the graph $\Gf(v, A, Z)$. 
The following theorem is proven in Appendix A in the supplement \citep{supplemental}.

\begin{thm} \label{thm:check-solution}
  Let $\Gl=(V \cup \cL, D)$ be a latent-factor graph, and fix a node $v \in V$, a set $H \subseteq \cL$ and a set $Z \subseteq \textrm{ch}(H) \setminus (\{v\} \cup \pa_V(v))$ such that $|Z|=|H|$. For the set of allowed nodes $A = V\setminus (Z \cup \{v\} \cup \mathrm{ch}(\pa_{\cL}(Z \cup \{v\}) \setminus H))$ we have that $\texttt{MaxFlow}(\Gf(v, A, Z)) = |\pa_V(v)| + |Z|$ if and only if there exists $Y \subseteq A$ such that the triple $(Y,Z,H)$ satisfies the LF-HTC for $v \in V$.
\end{thm}

For solving \texttt{LF-HTC}$(\Gl,v)$ we iterate over all suitable sets $H \subseteq \cL$ and $Z \subseteq \ch(H) \setminus (\{v\} \cup \pa_V(v))$ such that $|Z|=|H|$ and check for each pair $(Z,H)$ if there is a corresponding set $Y \subseteq A$. In each iteration, we have to compute one maximum flow by Theorem \ref{thm:check-solution}. It is enough to iterate over subsets $H \subseteq \cL_{\geq 4}$ where  $\cL_{\geq 4} = \{h \in \cL : |\ch(h)| \geq 4\}$ contains only those latent nodes with more than four children. Recall that the children of a node $v \in V \cup \cL$ are formally defined as $\ch(v) = \{w \in V \cup \cL : v \rightarrow w \in D\}.$   We prove the following fact in Appendix A in the supplement.

\begin{prop} \label{prop:four-children}
  Let $\Gl=(V \cup \cL, D)$ be a latent-factor graph, and fix a node $v \in V$. If the triple  $(Y,Z,H)$ satisfies the LF-HTC for $v\in V$ and there is a node $h \in H$ such that $|\textrm{ch}(h)| \leq 3$, then there are subsets $\widetilde{Y} \subseteq Y$ and $\widetilde{Z} \subseteq Z$ such that the triple $(\widetilde{Y}, \widetilde{Z} , \widetilde{H})$ with $\widetilde{H} = H \setminus \{h\}$  satisfies the LF-HTC for $v\in V$ as well.
\end{prop}

\begin{algorithm}[b]
\caption{Testing LF-HTC-identifiability of a latent-factor graph}\label{graph-algo}
\begin{algorithmic}[1]
\REQUIRE Latent-factor graph $\Gl = (V \cup \cL, D)$.\\
\ENSURE Solved nodes $S \leftarrow \{v \in V: \pa_V(v) = \emptyset\}$.
\REPEAT 
    \FOR {$v \in V \setminus S$}
        \FOR {$H \in \cL_{\geq 4}$}
            \STATE{$Z_a \leftarrow (S \cap \ch(H)) \setminus (\{v\} \cup \pa_V(v))$.}
            \FOR {$Z \subseteq Z_a$ such that $|Z| = |H|$} 
                \STATE{$A \leftarrow V \setminus (Z \cup \{v\} \cup \ch(\pa_{\cL}(Z \cup \{v\}) \setminus H) \cup (\htr_H(Z \cup \{v\}) \setminus S))$.}
                \IF {$\texttt{MaxFlow}(\Gf(v,A,Z))  = |\pa_V(v)| + |Z|$}
                     \STATE {$S \leftarrow S \cup \{v\}$}
                     \BREAK
                \ENDIF
            \ENDFOR
            \IF {$v \in S$}
                \BREAK
            \ENDIF
        \ENDFOR
    \ENDFOR
\UNTIL{$S = V$ or no change has occurred in the last iteration.}
\RETURN ``yes'' if $S=V$, ``no'' otherwise.
\end{algorithmic}
\label{alg:check-lfhtc}
\end{algorithm}

Next, we give an algorithm to determine whether a graph $\Gl$ is LF-HTC-identifiable by iterating over all nodes $v \in V$ and solving \texttt{LF-HTC}$(\Gl,v)$ in each step. Moreover, when solving \texttt{LF-HTC}$(\Gl,v)$ for a specific node $v \in V$, we have to make sure that, for a possible solution $(Y,Z,H)$, each node $w \in Z \cup (Y \cap \htr_H(Z \cup \{v\}))$ was solved before. This intuition is formalized in Algorithm \ref{alg:check-lfhtc}. In Theorem \ref{thm:algo-correct} we prove that the algorithm correctly determines LF-HTC-identifiability. Our implementation of Algorithm \ref{alg:check-lfhtc} is included in the \texttt{R} package \texttt{SEMID} as of version 0.4.0 \citep{R, semid}, which is available on \texttt{CRAN}, the Comprehensive R Archive Network.

\begin{thm} \label{thm:algo-correct}
  A latent-factor graph $\Gl = (V \cup \cL, D)$ is LF-HTC-identifiable if and only if Algorithm \ref{alg:check-lfhtc} returns ``yes''. If we only allow sets $H$ with $|H| \leq k$ in line $3$, then the algorithm has complexity at most $\mathcal{O}(|V|^{2+k} |\cL|^k (|V|+|\cL|+r)^3)$ where $r \leq |D_V|/2$ is the number of reciprocal edge pairs in $D_V$.
\end{thm}

In Algorithm \ref{alg:check-lfhtc} we iterate over subsets of the power sets of $\cL$ and $V$, and we put effort into iterating over a small subset. Nevertheless, if we allow the cardinality of $|H|$ to be unbounded in line three, then we search over an exponentially large space and, thus, our algorithm will in general take exponential time $\mathcal{O}(2^{|\cL|+|V|})$.  In fact, there is a fundamental barrier in finding a polynomial time algorithm as we are able to show that \texttt{LF-HTC}$(\Gl,v)$  is an NP-complete problem. 

To see that \texttt{LF-HTC}$(\Gl,v)$ is NP-complete, first note that \texttt{LF-HTC}$(\Gl,v)$ is in the NP-complexity class due to Theorem \ref{thm:check-solution}. Every candidate triple $(Y,Z,H)$ to solve \texttt{LF-HTC}$(\Gl,v)$ can be checked to be a solution in polynomial time by first checking if $(Y,Z,H)$ satisfies conditions (i) and (ii) of the LF-HTC and then checking if $\texttt{MaxFlow}(\Gf(v, Y, Z)) = |\pa_V(v)| + |Z|$.   Moreover, we are able to show NP-hardness of \texttt{LF-HTC}$(\Gl,v)$ by a reduction from  the Boolean satisfiability problem in conjunctive normal form; this result is developed in Appendix B in the supplement \citep{supplemental}.

\section{Numerical Experiments} \label{sec:experiments}

This section reports on the results of experiments with small latent-factor graphs, for which the identification problem can be fully solved by techniques from computational algebraic geometry, as we discuss in Appendix C in the supplement \citep{supplemental}. We study acyclic latent-factor graphs with $|V|=6$ observed nodes.

\begin{figure}[b]
    \centering
    \tikzset{
      every node/.style={circle, inner sep=0.3mm, minimum size=0.45cm, draw, thick, black, fill=white, text=black},
      every path/.style={thick}
    }
    \begin{tikzpicture}[align=center]
      \node[fill=lightgray] (h1) at (0,1.3) {$h_{1}$};
      \node[] (1) at (-2.5,0) {};
      \node[] (2) at (-1.5,0) {};
      \node[] (3) at (-0.5,0) {};
      \node[] (4) at (0.5,0) {};
      \node[] (5) at (1.5,0) {};
      \node[] (6) at (2.5,0) {};
      
      \draw[red, dashed] [-latex] (h1) edge (1);
      \draw[red, dashed] [-latex] (h1) edge (2);
      \draw[red, dashed] [-latex] (h1) edge (3);
      \draw[red, dashed] [-latex] (h1) edge (4);
      \draw[red, dashed] [-latex] (h1) edge (5);
      \draw[red, dashed] [-latex] (h1) edge (6);
      
    \end{tikzpicture}
    \caption{Latent structure of unlabeled latent-factor graph with one global latent factor.}
    \label{fig:exp-one-latent}
\end{figure}
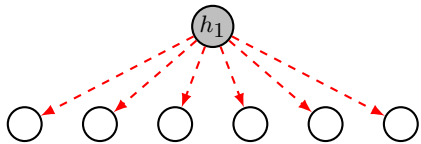

\begin{table*}[b]\centering 
\begin{tabular}{@{}c|rrrr@{}}\toprule
\shortstack{nr of obs. \\ edges $|D_V|$} & total & \shortstack{generically \\ finite-to-one} & \shortstack{rationally \\ identifiable} & \shortstack{LF-HTC- \\ identifiable} \\ 
\midrule
0 & 1 & 1 & 1 & 1  \\
1 & 1 & 1 & 1 & 1  \\
2 & 4 & 4 & 4 & 4  \\
3 & 13 & 13 & 13 & 13  \\
4 & 51 & 51 & 51 & 50  \\
5 & 163 & 160 & 159 & 134  \\
6 & 407 & 401 & 398 & 250  \\
7 & 796 & 770 & 747 & 234  \\
8 & 1169 & 1047 & 956 & 64  \\
9 & 1291 & 896 & 631 & 4  \\
\midrule
Total & 3896 & 3344 & 2961 & 755 \\
\bottomrule
\end{tabular}
\caption{Counts of unlabeled DAGs with $|V|=6$ observed nodes and one latent node as in Figure \ref{fig:exp-one-latent}.}
\label{table:exp-one-latent}
\end{table*}

In the first experimental setup we consider one global latent factor that has an effect on all observed variables, as illustrated in Figure \ref{fig:exp-one-latent}.  All possible DAGs on 6 nodes are considered for the observed part $(V, D_V)$.  Table \ref{table:exp-one-latent} lists the counts when there are $|D_V|\le 9$ edges in the observed part of the graph.  Graphs with $|D_V| > 9$ are trivially generically infinite-to-one by Corollary \ref{cor:infinite-to-one}. In the counts in Table \ref{table:exp-one-latent} we treat graphs as unlabeled, that is, we count isomorphism classes of graphs. Formally, two latent-factor graphs $G=(V \cup \cL, D)$ and $G^{\prime}=(V \cup \cL, D^{\prime})$ with the same set of nodes  are isomorphic if there is a permutation $\pi$ of the observed nodes $V$ such that for two nodes $h \in \cL$ and $ v \in V$ the edge $h \rightarrow v \in D$ if and only if $h \rightarrow \pi(v) \in D^{\prime}$ and for two nodes $v,w \in V$ the edge $v \rightarrow w \in D$ if and only if $\pi(v) \rightarrow \pi(w) \in D^{\prime}$.

In the second setup we consider two latent factors, each of them only having influence on some of the observed variables. The precise latent structure is illustrated in Figure \ref{fig:exp-two-latent}.  Since the number of isomorphism classes is much larger in this case, for computational reasons we only consider graphs with at most $|D_V|=6$ edges between observed nodes.  Up to this constraint, the observed part may be any DAG.  Table \ref{table:exp-two-latent} lists the counts for these graphs, again up to isomorphism. In this setup it is possible that the latent projection is rationally identifiable. Thus, we compare the LF-HTC with the original HTC applied to the projection and the results are counted in an additional column.

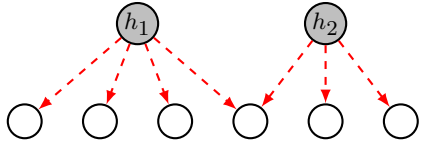
\begin{figure}[t]
    \centering
    \tikzset{
      every node/.style={circle, inner sep=0.3mm, minimum size=0.45cm, draw, thick, black, fill=white, text=black},
      every path/.style={thick}
    }
    \begin{tikzpicture}[align=center]
      \node[fill=lightgray] (h1) at (-1,1.3) {$h_{1}$};
      \node[fill=lightgray] (h2) at (1.5,1.3) {$h_{2}$};
      \node[] (1) at (-2.5,0) {};
      \node[] (2) at (-1.5,0) {};
      \node[] (3) at (-0.5,0) {};
      \node[] (4) at (0.5,0) {};
      \node[] (5) at (1.5,0) {};
      \node[] (6) at (2.5,0) {};
      
      \draw[red, dashed] [-latex] (h1) edge (1);
      \draw[red, dashed] [-latex] (h1) edge (2);
      \draw[red, dashed] [-latex] (h1) edge (3);
      \draw[red, dashed] [-latex] (h1) edge (4);
      \draw[red, dashed] [-latex] (h2) edge (4);
      \draw[red, dashed] [-latex] (h2) edge (5);
      \draw[red, dashed] [-latex] (h2) edge (6);
      
    \end{tikzpicture}
    \caption{Latent structure of unlabeled latent-factor graphs with two latent factors.}
    \label{fig:exp-two-latent}
\end{figure}

\begin{table*}[t]\centering 
\begin{tabular}{@{}c|rrrrr@{}}\toprule
\shortstack{nr of obs. \\ edges $|D_V|$} & total & \shortstack{generically \\ finite-to-one} & \shortstack{rationally \\ identifiable} & \shortstack{LF-HTC- \\ identifiable} & \shortstack{HTC- \\ identifiable}\\ 
\midrule
0 & 1 & 1 & 1 & 1 & 1\\
1 & 8 & 6 & 6 & 6 & 4 \\
2 & 63 & 45 & 45 & 43 & 24  \\
3 & 391 & 255 & 255 & 236 & 104 \\
4 & 1983 & 1171 & 1171 & 1018 & 384  \\
5 & 7570 & 3907 & 3898 & 3028 & 900  \\
6 & 21029 & 9080 & 8960 & 5861 & 1157 \\
\midrule
Total & 31045 & 14465 & 14336 & 10193 & 2574 \\
\bottomrule
\end{tabular}
\caption{Counts of unlabeled DAGs with $|V|=6$ observed nodes and two latent nodes as in Figure \ref{fig:exp-two-latent}.}
\label{table:exp-two-latent}
\end{table*}

In the considered setups, we see that the latent factor-criterion is very successful in certifying the graphs to be rationally identifiable as long as the number of observed edges $|D_V|$ is not too large. It misses more graphs the larger the number of observed edges is. Moreover, in the second setup, the latent-factor half-trek criterion  declares about four times more graphs to be rationally identifiable than the original half-trek criterion applied to the latent projection.

\section{Proof of main result}
\label{sec:proof-lfhtc}
In this section we prove the main theorem.

\begin{proof}[Proof of Theorem \ref{thm:lfhtc}]
  Let $\pa_V(v) = \{p_1,\dots,p_n\}$, $H\subseteq \cL$ with $|H|=r$, $Y = \{y_1,\dots,y_{n+r}\}$, and $Z = \{z_1,\dots,z_r\}$ be as in the statement of the theorem. Define matrices $A\in\bR^{(n+r)\times n},B\in\bR^{(n+r)\times r}$ and a vector $c\in\bR^{n+r}$ as follows:
\[A_{ij} = \begin{cases}
\left[(I_d - \L)^\top\Sigma\right]_{y_ip_j},&\text{ if }y_i\in\htr_{H}(Z\cup \{v\}),\\
\Sigma_{y_ip_j},&\text{ if }y_i\not\in\htr_{H}(Z\cup\{v\}),
\end{cases}\]
and
\[B_{ij} = \begin{cases}
\left[(I_d - \L)^\top\Sigma(I_d - \L)\right]_{y_iz_j},&\text{ if }y_i\in\htr_{H}(Z\cup \{v\}),\\
\left[\Sigma(I_d - \L)\right]_{y_iz_j},&\text{ if }y_i\not\in\htr_{H}(Z\cup\{v\}),
\end{cases}\]
and
\[c_i = \begin{cases}
\left[(I_d - \L)^\top\Sigma\right]_{y_iv},&\text{ if }y_i\in\htr_{H}(Z\cup\{v\}),\\
\Sigma_{y_iv},&\text{ if }y_i\not\in\htr_{H}(Z\cup\{v\}).
\end{cases}\]

\textbf{Claim $1$.} The matrices $A$ and $B$ and the vector $c$ are
all rationally identifiable.  \\[0.1cm]
By assumption, all columns of $\L$ indexed by a vertex in $Z\cup (Y\cap \htr_{H}(Z\cup\{v\}))$ are rationally identifiable (i.e., rational functions of $\Sigma$).  Inspecting the above expressions, we observe that only entries from these columns of $\L$ appear in the definition of $A$, $B$, and $c$.  Hence, $A$, $B$, and $c$ are rationally identifiable, as claimed.

\smallskip

Next, note that there is a set $Y_Z \subseteq Y$ such that there is a system of latent-factor half-treks with no sided intersection from $Y_Z$ to $Z$. In this system each half-trek takes the form $y \leftarrow h \rightarrow z$ for $y \in Y$, $z \in Z$ and $h \in H$. Since the system has no sided intersection, it follows from Proposition 3.4 in \citet{sullivant2010trek} that $\det(\Omega_{Y_Z, Z}) \neq 0$ generically. Thus the matrix $\Omega_{Y, Z}$ has full column rank $r$ because $\Omega_{Y_Z, Z}$ is a submatrix. Using this fact we prove our next claim. 
\smallskip

\textbf{Claim $2$.} There exists some $\psi\in\bR^r$ such that
\begin{align*}
  \begin{pmatrix} A & B \end{pmatrix}\cdot \begin{pmatrix} \L_{\pa_V(v),v} \\ \psi\end{pmatrix} = c\,.
\end{align*}

\noindent
To see this, we will implicitly construct $\psi$.  Let $\O_h = \Gamma_h^{\top} \Gamma_h$ for each $h \in \cL$, and observe that
\begin{align*}
  \O_{Y,Z\cup\{v\}} = (\Od)_{Y,Z\cup\{v\}} + \sum_{h\in H}(\O_{h})_{Y,Z\cup\{v\}} + \sum_{h\in\cL\backslash H}(\O_{h})_{Y,Z\cup\{v\}}\;.
\end{align*}
Since $Y\cap (Z\cup\{v\})=\emptyset$ by definition of the
latent-factor half-trek criterion, we have that
$ (\Od)_{Y,Z\cup\{v\}} =0$.  The definition of the
latent-factor half-trek criterion yields furthermore that for any $h\in\cL\setminus H$, either $Y\cap \ch(h)=\emptyset$ or
$(Z\cup\{v\})\cap \ch(h) =\emptyset$.  Hence,
$(\O_{h})_{Y,Z\cup\{v\}}=0$.  We obtain that
\begin{align*}
  \O_{Y,Z\cup\{v\}} =  \sum_{h\in H}(\O_{h})_{Y,Z\cup\{v\}} = (\O_{H})_{Y,Z\cup\{v\}}\;,
\end{align*}
where $\O_{H} := \sum_{h\in H}\O_h$. Note that
$\rank(\O_H)\leq |H| = r$.  Moreover, $(\O_H)_{V,Z}$ has full column
rank $r$ by assumption (since $\O_{Y,Z}$ is a submatrix of this
matrix), which proves that
\begin{equation}\label{eq:def-psi}
    (\O_H)_{V,Z} \cdot \psi = (\O_H)_{V,v}
\end{equation}
for some $\psi\in\bR^r$.

Next, consider any index $i$ such that $y_i\in\htr_{H}(Z\cup\{v\})$. Then
\begin{align}
\nonumber
&\left[\begin{pmatrix} A & B \end{pmatrix}\cdot \begin{pmatrix} \L_{\pa_V(v),v} \\ \psi\end{pmatrix} \right]_i\\
\nonumber
&=\left[(I_d - \L)^\top\Sigma\right]_{y_i,\pa_V(v)}\cdot \L_{\pa_V(v),v} + \left[(I_d - \L)^\top\Sigma(I_d - \L)\right]_{y_i,Z}\cdot \psi\\
\label{eq:id-equation}
&=\left[(I_d - \L)^\top\Sigma\cdot \L\right]_{y_iv} + \left[\O_{Y,Z}\cdot \psi\right]_i
\end{align}
because $\L_{wv}=0$ unless $w\in\pa_V(v)$ and $ (I_d - \L)^\top\Sigma(I_d - \L)=\O $.  Since $\O_{Y,Z\cup\{v\}} = (\O_{H})_{Y,Z\cup \{v\}}$, it follows from \eqref{eq:def-psi} that
\[
\left[\O_{Y,Z}\cdot \psi\right]_i = \left[\O_{Y,v}\right]_i = \O_{y_iv}.
\]
Hence, we may rewrite \eqref{eq:id-equation} as
\begin{align*}
\left[\begin{pmatrix} A & B \end{pmatrix}\cdot \begin{pmatrix} \L_{\pa_V(v),v} \\ \psi\end{pmatrix} \right]_i
&=\left[(I_d - \L)^\top\Sigma\right]_{y_iv} - \left[(I_d - \L)^\top\Sigma (I_d - \L)\right]_{y_iv} + \O_{y_iv}\\
&=\left[(I_d - \L)^\top\Sigma\right]_{y_iv} - \O_{y_iv} + \O_{y_iv}\\
&=c_i, 
\end{align*}
by the definition of $c$.

To conclude the proof of Claim 2, consider any index $i$ such that $y_i\not\in\htr_{H}(Z\cup\{v\})$.  For any such $i$, any latent-factor half-trek from a node $w\in Z\cup\{v\}$ to $y_i$ must be of the form
\[w \leftarrow h \rightarrow x_1 \rightarrow \dots \rightarrow x_m \rightarrow y_i\]
for some $h\in H$. This implies that
\begin{equation}\label{eqn:not_htr}\left[\O (I_d - \L)^{-1}\right]_{wy_i} = \left[\O_{H} (I_d - \L)^{-1}\right]_{wy_i}\end{equation}
for all $w\in Z\cup\{v\}$.  Consequently,
\begin{align}
\nonumber
\left[\begin{pmatrix} A & B \end{pmatrix}\cdot \begin{pmatrix} \L_{\pa_V(v),v} \\ \psi\end{pmatrix} \right]_i
&=\Sigma_{y_i,\pa(v)}\cdot \L_{\pa_V(v),v} + \left[\Sigma(I_d - \L)\right]_{y_i,Z}\cdot \psi\\
\nonumber
&=\left[\Sigma \L\right]_{y_i v} + \left[\Sigma(I_d - \L)\right]_{y_i,Z}\cdot \psi\\
\nonumber
&=\Sigma_{y_i v} - \left[\Sigma (I_d - \L)\right]_{y_i v} + \left[\Sigma(I_d - \L)\right]_{y_i,Z}\cdot \psi\\
\label{eq:claim2_3}
&=\Sigma_{y_i v} - \left[(I_d - \L)^{-\top}\O\right]_{y_i v} + \left[(I_d - \L)^{-\top}\O\right]_{y_i,Z}\cdot \psi,
\end{align}
because $\O = (I_d - \L)^\top\Sigma(I_d - \L)$.  Applying first~\eqref{eqn:not_htr} and then~\eqref{eq:def-psi}, we find that
\begin{align*}
- \left[(I_d - \L)^{-\top}\O\right]_{y_i v} &+ \left[(I_d - \L)^{-\top}\O\right]_{y_i,Z}\cdot \psi \\
&=- \left[(I_d - \L)^{-\top}\O_{H}\right]_{y_i v} + \left[(I_d - \L)^{-\top}\O_{H}\right]_{y_i,Z}\cdot \psi\\
&=- \left[(I_d - \L)^{-\top}\O_{H}\right]_{y_i v} + \left[(I_d - \L)^{-\top}\O_{H}\right]_{y_iv} \qquad=\; 0.
\end{align*}
Taking up \eqref{eq:claim2_3} and recalling the definition of $c$, we conclude that
\begin{align*}
\left[\begin{pmatrix} A & B \end{pmatrix}\cdot \begin{pmatrix} \L_{\pa_V(v),v} \\ \psi\end{pmatrix} \right]_i
&=\Sigma_{y_i v} = c_i.
\end{align*}

The theorem is now proven if the equation system exhibited in Claim 2 has a unique solution generically.  This is addressed by our last claim:
\smallskip

\textbf{Claim $3$.} The matrix
$\begin{pmatrix} A & B \end{pmatrix}$
is generically invertible.
\smallskip

\noindent
To prove Claim $3$, we will show that if we set some parameters equal to zero, then the considered matrix is invertible for generic choices of the remaining free parameters, which is sufficient to show that the matrix will be generically invertible with respect to choices of all parameters.

By assumption, the latent-factor graph $\Gl$ contains a system of latent-factor half-treks from $Y$ to $Z\cup\pa_V(v)$, where half-treks terminating at any $z\in Z$ are of the form $y_i \leftarrow h \rightarrow z$ for some $h\in H$. 
 For every $z\in Z$, set $\L_{\pa_V(z),z}=0$. Furthermore, every node $h\in H$ appears in at most one of the latent-factor half-treks in the system. Suppose it appears as
$y_i \leftarrow h \rightarrow w$. Then we will define $\Omega_h$ to have value $\omega_{y_iw}$ at entries $\{y_i,w\}\times \{y_i,w\}$, and zeros elsewhere. 

Consider now a mixed graph $\widehat{G}$ constructed as follows. Starting with the induced subgraph $\widehat{G}=(V,D_V)$, first remove all edges with head in $Z$. Next, looking at the selected system of latent-factor half-treks from $Y$ to $Z\cup \pa_V(v)$ in the latent-factor graph $\Gl$, any time we see a half-trek beginning with $y_i \leftarrow h\rightarrow w$, add a bidirected edge $y_i\leftrightarrow w$ to $\widehat{G}$.

By definition of the new graph $\widehat{G}$, the selected system of latent-factor half-treks from $Y$ to $Z\cup \pa_V(v)$ in $\Gl$ has a corresponding system of half-treks in $\widehat{G}$.  Here, any latent-factor half-trek that begins with edges $y_i \leftarrow h \rightarrow w$ has these two initial two edges replaced by the bidirected edge $y_i\leftrightarrow w$.  The resulting system of half-treks in $\widehat{G}$ has no sided intersection. Let $\widehat{\L}$ and $\widehat{\O}$ be the parameter matrices for this graph. Note that $(I - \widehat{\L})_{*,Z} = I_{*,Z}$ because $\widehat{\L}_{*,Z}=0$ by construction. Therefore, we can write
\[B_{ij} = \begin{cases}
\left[(I_d - \L)^\top\Sigma\right]_{y_iz_j},&\text{ if }y_i\in\htr_{H}(Z\cup\{v\}),\\
\Sigma_{y_iz_j},&\text{ if }y_i\not\in\htr_{H}(Z\cup\{v\}).
\end{cases}\]
We now apply Lemma 2 in the original half-trek paper \citep{foygel2012halftrek} to conclude that $\begin{pmatrix} A & B \end{pmatrix}$
is generically invertible.
\end{proof}

\section{Discussion}
In this work we proposed a graphical criterion that provides an effective sufficient condition for rational identifiability in linear structural equation models where latent variables are not projected to correlation among noise terms. To the best of our knowledge, it is the most general graphical criterion to decide identifiability for graphs explicitly including latent nodes. The new criterion can be checked in time that is polynomial in the size of the graph if we search only over subsets of latent nodes of bounded size. 
The restriction of the search space is necessary since checking the criterion without any restriction is in general NP-hard.

The criterion applies to a wide range of models and allows for presence of multiple latent factors that may even have an effect on many or all of the observed variables. The corresponding directed graph is allowed to be cyclic, the only restriction that we made in this work is that all latent factors are source nodes in the graph. 

It is noteworthy that even if a model is not LF-HTC-identifiable, the latent-factor half-trek method can still prove certain columns of $\Lambda$ to be identifiable. This is the case if the recursive procedure of Algorithm \ref{alg:check-lfhtc} stops early declaring some but not all nodes to satisfy the LF-HTC.  In this case, the status of identifiability of the whole graph remains inconclusive but for the nodes $v$ that the method successfully visits, the parameters $\Lambda_{\pa_V(v), v}$ are proven to be rationally identifiable. 

Methods for identifiability of latent-factor graphs are useful also as a refinement  
of
methods that operate on mixed graphs in the latent projection framework: Imagine a model that is generically infinite-to-one in the latent projection framework.  The main reason for this is often denser  confounding, that is, there is confounding between many of the observed variables.  There is then the natural question whether the model would be (rationally) identifiable if the confounding originated from a simpler structure, i.e., is caused by only a few latent factors. Then the LF-HTC may be applicable and may prove a model rationally identifiable. On the other hand, if a model is rationally identifiable in the latent projection framework, then the identifiability may be due to the assumption that confounding is caused by multiple different latent factors. As shown in Figure \ref{fig:counterexmp}, there may be settings where rational identifiability no longer holds when the confounding is in fact caused by fewer factors. Using our method it is possible to check for such identifiability failures. 
    
We would like to emphasize that the LF-HTC is useful also if the goal is model selection.  One may then be interested in testing the goodness-of-fit of a particular model, a problem for which it 
is crucial to know the dimension of the model.  The LF-HTC asserting identifiability also means that the model has the expected dimension obtained from counting parameters.

\medskip

An interesting research program emerges from the work presented here.  Indeed, one may strive to improve and extend the efficiency of the LF-HTC along similar lines as those that have been applied in previous work that has led to improvements of the original half-trek criterion for mixed graphs. 
In particular, it would be useful to find a latent-factor modification of the criterion for edgewise identifiability that allows for identification of a subset or even single direct causal effects $\lambda_{wv}$ instead of only targeting whole columns $\Lambda_{\pa_V(v), v}$; compare to \citet{weihs2017determinantal} and references therein.  This extension is of interest when effects between particular variables are the primary targets of investigation, but it may also make the criterion more powerful as a whole. Another way to extend the scope of the LF-HTC would be to apply graph decomposition techniques as proposed by \citet{tian2005identifying}; see also \citet{foygel2012halftrek} and \citet[Section 6]{drton:2018}. 

Furthermore, it would be interesting to generalize the LF-HTC to a version in which we relax the condition that all latent factors are source nodes in the graph. For example, one may consider models where latent nodes are only required to be \emph{upstream}, i.e., there may be direct causal effects between latent variables but no effects from observed variables to latent variables. Put differently, in addition to the equation system \eqref{eq:def-model} that defines the model, the vector of latent variables  $(L_h)_{h\in\mathcal{L}}$ is required to satisfy the equation
$$
    L = B^{T}L + \delta
$$
where $B$ is an $\ell \times \ell$ matrix with zeros along the diagonal and the noise terms $\delta = (\delta_h)_{h\in\mathcal{L}}$ are independent with mean zero and variance $1$. The latent covariance matrix is now of the form
$$
    \Omega = \Od + \Gamma^{\top} (I_{\ell} - B)^{-\top} (I_{\ell} - B)^{-1} \Gamma.
$$
Thus the parametrization $\tau$ of the cone of latent covariance matrices is rational and depends on the three parameter matrices $(B, \Gamma, \Od)$. The question is how to identify effects between observed variables in this case, or, even more, what can be said in terms of identifying causal effects between latent variables. Note that such a setting cannot be handled by a mixed graph approach which marginalizes out the effects of interest. Hence our work sets the scene for future developments of identifiability between latent variables.

In Lemma \ref{lem:finite-to-one} we gave a simple necessary condition for the parametrization map to be generically finite-to-one.  In future work, we hope to obtain more powerful necessary conditions for generic identifiability in the form of efficient graphical criteria.  This will amount to studying the Jacobian matrix of the parametrization $\phi_{\Gl}$, taking into account the algebraic geometry of the cone of latent covariance matrices.


\begin{acks}[Acknowledgments]
This project has received funding from the European Research Council (ERC) under the European Union’s Horizon 2020 research and innovation programme (grant agreement No 883818). Nils Sturma acknowledges support by the Munich Data Science Institute (MDSI) at Technical University of Munich (TUM) via the Linde/MDSI PhD Fellowship program.  Rina Foygel Barber was supported by the U.S.~National Science Foundation via grants DMS-1654076 and DMS-2023109, and by the Office of Naval Research via grant N00014-20-1-2337.
\end{acks}

\begin{supplement}
\stitle{Supplement to ``Half-Trek Criterion for Identifiability of Latent Variable Models''}
\sdescription{The supplement contains additional material such as further elements of proofs, a hardness result for checking the LF-HTC without a bound on the cardinality of searched sets of latent variables, and an explanation on how to effectively deploy techniques from computational algebraic geometry.}
\end{supplement}


\bibliographystyle{imsart-nameyear} 
\bibliography{literature}       


\end{document}


\sloppy 

\begin{frontmatter}
\title{Supplement to ``Half-Trek Criterion for Identifiability of Latent Variable Models''}
\runtitle{Identifiability in Latent Variable Models}

\begin{aug}
\author[A]{\fnms{Rina Foygel} \snm{Barber}\ead[label=e1]{rina@uchicago.edu}},
\author[B]{\fnms{Mathias} \snm{Drton}\ead[label=e2,mark]{mathias.drton@tum.de}},
\author[B]{\fnms{Nils} \snm{Sturma}\ead[label=e3,mark]{nils.sturma@tum.de}}
\and
\author[C]{\fnms{Luca} \snm{Weihs}\ead[label=e4]{lucaw@allenai.org}}
\address[A]{Department of Statistics, University of Chicago, \printead{e1}}
\address[B]{Department of Mathematics and Munich Data Science Institute, Technical University of Munich, \printead{e2,e3}}
\address[C]{Allen Institute for AI, \printead{e4}}
\end{aug}

\begin{abstract}
This supplement contains additional material such as further elements of proofs (Appendix \ref{sec:proofs}), a hardness result for checking the LF-HTC without a bound on the cardinality of searched sets of latent variables (Appendix \ref{sec:NP-completeness}), and an explanation on how to effectively deploy techniques from computational algebraic geometry (Appendix \ref{sec:algebraic-techniques}).
\end{abstract}

\begin{keyword}[class=MSC]
\kwd{62H22}
\kwd{62J05}
\kwd{62R01}
\end{keyword}

\begin{keyword}
\kwd{Covariance matrix}
\kwd{factor analysis}
\kwd{graphical model}
\kwd{hidden variables}
\kwd{latent variables}
\kwd{parameter identification}
\kwd{structural equation model}
\end{keyword}

\end{frontmatter}

\begin{appendix}

\section{Proofs} \label{sec:proofs}

\begin{proof}[Proof of Lemma 2.5]
Throughout the proof we let $p = \dim(S)$. Since $f$ is rational, it is a semialgebraic mapping according to Definition 2.2.5 in \citet{bochnak1998real}. Images and preimages of semialgebraic sets under semialgebraic mappings are again semialgebraic. Hence, the image $f(S)$ is a semialgebraic set. The rest of the proof is an application of Hardt's triviality theorem \citep[Theorem 5.45]{basu2006algorithms} which states that there exists a finite partition of $f(S)$ into semialgebraic sets $f(S) = \bigcup_{i=1}^r T_i$ such that for each $i$ and for each $y \in T_i$ the product $T_i \times f^{-1}(y)$ is semialgebraically homeomorphic to $f^{-1}(T_i)$.  In particular, we have for all $y \in T_i$ the equality
\begin{equation} \label{eq:dim-fiber}
    \dim(f^{-1}(y)) = \dim(f^{-1}(T_i)) - \dim(T_i).
\end{equation}

Now suppose that $k = \dim(f(S)) < p$. Observe that $S = \bigcup_{i=1}^r f^{-1}(T_i)$ is a finite union of semi-algebraic sets.  We write $C$ for the union of all preimages $f^{-1}(T_i)$ of dimension strictly less than $p$. Then for all $x \in S \setminus C$ we have by Equation \eqref{eq:dim-fiber}
$$
    \dim(f^{-1}(f(x))) \geq \dim(S) - \dim(f(S)) = p - k > 0,
$$
which means that for all  $x \in S \setminus C$ the fiber $f^{-1}(f(x))$ is a semialgebraic subset of $S$ with positive dimension, i.e., it contains infinitely many elements (cf.~Theorem 5.19 in \citet{basu2006algorithms}). Moreover, the Zariski closure $\overline{S}$ is equal to the union of Zariski closures $\overline{S \setminus C} \cup \overline{C}$. By Proposition 2.8.5 in \citet{bochnak1998real} the dimension of $C$ is strictly less than $p$, i.e., $\overline{S} \neq \overline{C}$. Since $\overline{S}$ is irreducible, it must be the case $\overline{S} = \overline{S \setminus C}$. Thus there is no proper algebraic subset of $\overline{S}$ that contains $S \setminus C$ and we conclude that $f$ is generically infinite-to-one.

For the other direction, suppose that $k = \dim(f(S)) = p$. Let $I = \{i \in \{1, \ldots, r\}: \dim(T_i) < p\}$ and $B = \bigcup_{i \in I} T_i$. Then the Zariski closure $\overline{B}$ in $\mathbb{R}^n$ has dimension strictly smaller than $p$.
Applying Equation \eqref{eq:dim-fiber} we get for all $y \in f(S) \setminus \overline{B}$ that
$$
    \dim(f^{-1}(y)) \leq \dim(S) - p = p - p = 0.
$$
Therefore, for all $x \in S\setminus f^{-1}(\overline{B})$ the fiber $f^{-1}(f(x))$ is a zero-dimensional and thus finite semialgebraic set (compare Theorem 5.19 in \citet{basu2006algorithms} again).
To finish the proof it remains to show that the Zariski closure of $f^{-1}(\overline{B})$ is a proper subset of $\overline{S}$.  As $\overline{S}$ is assumed to be irreducible, it suffices to argue that $\overline{S}$ contains a point outside the Zariski closure of $f^{-1}(\overline{B})$.  Using that $f$ is rational, we see that the preimage $f^{-1}(\overline{B})$ is an algebraic subset of $S$. Since $\dim(\overline{B}) < p$, the set $f(S) \setminus \overline{B}$ is nonempty and therefore $S \setminus f^{-1}(\overline{B})$ is nonempty as well.  Now observe that the points in $S \setminus f^{-1}(\overline{B})$ are not contained in the Zariski closure of $f^{-1}(\overline{B})$.  We conclude that $f$ is generically finite-to-one.
\end{proof}

\begin{proof}[Proof of Theorem 5.1]
The proof is similar to the proof of Theorem $6$ in \citet{foygel2012halftrek}. If $(Y,Z,H) \in 2^{V\setminus\{v\}}\times 2^{V\setminus \{v\}}\times 2^\cL$ satisfies the LF-HTC with respect to $v$, then we have a system $\Pi$ of latent-factor half-treks from $Y$ to $\pa_V(v) \cup Z$ with no sided intersection such that for each $z\in Z$, the half-trek terminating at $z$ takes the form $y\leftarrow h \rightarrow z$ for some $y\in Y$ and some $h\in H$.
  
For each latent-factor half-trek $\pi_k \in \Pi$ of the form 
$$
  \pi_k: y_k \leftarrow h_k \rightarrow w_k \rightarrow \cdots \rightarrow k, \quad k \in \pa_V(v) \cup Z,
$$
add a flow of size $1$ along the path
$$
\widetilde{\pi}_k: s \rightarrow y_k \rightarrow h_k \rightarrow h_k^{\prime} \rightarrow w_k^{\prime} \rightarrow \cdots \rightarrow k^{\prime} \rightarrow t
$$
in the flow graph $\Gf$. Similarly, for each latent-factor half-trek $\pi_k \in \Pi$ of the form 
$$
   \pi_k: y_k \rightarrow w_k \rightarrow \cdots \rightarrow k, \quad k \in \pa_V(v) \cup Z,
$$
add a flow of size $1$ along the path
$$
   \widetilde{\pi}_k: s \rightarrow y_k \rightarrow y_k^{\prime} \rightarrow w_k^{\prime} \rightarrow \cdots \rightarrow k^{\prime} \rightarrow t
$$
in the flow graph $\Gf$. Let $\widetilde{\Pi} = \{\widetilde{\pi}_k: k \in \pa_V(v) \cup Z\}$ be the system of directed paths that we obtain in the flow graph $\Gf$. Clearly, the total flow size from $s$ to $t$ in the flow graph is $|\pa_V(v)| + |Z|$. It is left to check that no capacity constraint is exceeded. This is trivial for the infinite edge capacities as well as for the infinite capacities of the nodes $s$ and $t$. For all other nodes that appear in some of the paths of the system $\widetilde{\Pi}$, note that they appear exactly once in the system since the original system of latent-factor half-treks $\Pi$ has no sided intersection.

Now suppose $\texttt{MaxFlow}(\Gf(v, A, Z)) = |\pa_V(v)| + |Z|$. By the properties of the max-flow problem with integer-valued capacities \citep{ford1962flows}, this means that there are $|\pa_V(v)| + |Z|$ directed paths from $s$ to $t$ with flow size $1$ along each path. We denote the collection of these paths by $\widetilde{\Pi}=\{\widetilde{\pi}_k: k \in \pa_V(v) \cup Z\}$, recall that by assumption $Z \cap \pa_V(v) = \emptyset$. Since all nodes in the flow graph that are not equal to $s$ or $t$ have capacity $1$, each node different from $s$ and $t$ can appear at most once in the system of paths $\widetilde{\Pi}$. Consider a specific path $\widetilde{\pi}_k \in \widetilde{\Pi}$. By construction of the graph $\Gf$, it has one of two forms.  First, we may have
$$
  \widetilde{\pi}_k: s \rightarrow y_k \rightarrow h_k \rightarrow h_k^{\prime} \rightarrow w_k^{\prime} \rightarrow \cdots \rightarrow k^{\prime} \rightarrow t
$$
with $y_k \in A$, $k \in \pa_V(v) \cup Z$ and $h_k \in \cL$. This defines the latent-factor half-trek
$$
\pi_k: y_k \leftarrow h_k \rightarrow w_k \rightarrow \cdots \rightarrow k
$$
in $\Gl$. The other possibility is that the path has the form 
$$
   \widetilde{\pi}_k: s \rightarrow y_k \rightarrow y_k^{\prime} \rightarrow w_k^{\prime} \rightarrow \cdots \rightarrow k^{\prime} \rightarrow t
$$
with $y_k \in A$ and $k \in  \pa_V(v) \cup Z$. This defines the latent-factor half-trek
$$
  \pi_k: y_k \rightarrow w_k \rightarrow \cdots \rightarrow k
$$
in $\Gl$.  In this way, we obtain a system of latent-factor half-treks $\Pi = \{\pi_k: k \in \pa_V(v) \cup Z\}$ in $\Gl$.  Because each node other than $s$ or $t$ appears at most once in the system $\widetilde{\Pi}$ in $\Gf$, the constructed system $\Pi$  has no sided intersection. Furthermore, if $k \in Z$, we have that $w_k = k$ in the latent-factor half-trek $\pi_k$ since by construction the flow graph $\Gf(v,A,Z)$ does not contain the edge $w^{\prime} \rightarrow z^{\prime}$ if $w \in Z$. Moreover, if $k \in Z$, it must be the case that $h_k \in H$. Indeed, if we have $h_k \notin H$, then $y_k \in \textrm{ch}(\pa_{\cL}(Z \cup \{v\}) \setminus H)$ which is impossible by assumption since $y_k \in A$.  Thus, $\Pi$ is a system of latent-factor half-treks with no sided intersection from $Y = \{y_k: k \in \pa_V(v) \cup Z\}$ to $Z\cup \pa_V(v)$ in $\Gl$, such that for each $z\in Z$, the half-trek terminating at $z$ takes the form $y\leftarrow h \rightarrow z$ for some $y\in Y$ and some $h\in H$. Finally, note that for the triple $(Y,Z,H)$ conditions (i) and (ii) of the LF-HTC are trivially satisfied by construction and the fact that $Y \subseteq A$.
\end{proof}

\begin{proof}[Proof of Proposition 5.2]
  Suppose the triple  $(Y,Z,H)$ satisfies the LF-HTC for $v\in V$ in $\Gl$. Recall that there exists a system of latent-factor half-treks $\Pi$ with no sided intersection from $Y$ to $\pa_V(v) \cup Z$ such that, for each $z\in Z$, the half-trek terminating at $z$ takes the form $y\leftarrow h \rightarrow z$ for some $y\in Y$ and some $h\in H$. Since $|Z| = |H|$, it is clearly not possible that there is a node $h \in H$ such that $|\ch(h)|=1$.
  
  Now let $h \in H$ such that $|\ch(h)| \in \{2,3\}$.  Then there is a unique latent-factor half-trek in $\Pi$ that has the form $y\leftarrow h \rightarrow z$ for some $y\in Y$ and some $z \in Z$. Let $\widetilde{Y} = Y \setminus \{y\}$ and $\widetilde{Z} = Z \setminus \{z\}$. It is clear that the triple $(\widetilde{Y}, \widetilde{Z} , \widetilde{H})$ satisfies conditions (i) and (iii) of the LF-HTC  and  $\widetilde{Y} \cap (\widetilde{Z} \cup \{v\}) = \emptyset$. Thus it is left to show that $h \not\in \pa_{\cL}(\widetilde{Y}) \cap \pa_{\cL}(\widetilde{Z} \cup \{v\})$.
  
  If $|\ch(h)|=2$, there are no more children of $h$ other than $y$ and $z$. Thus, we directly see that $h \not\in \pa_{\cL}(\widetilde{Y}) \cap \pa_{\cL}(\widetilde{Z} \cup \{v\})$. If $|\ch(h)|=3$, there might be one child $w \in \ch(h) \setminus \{y,z\}$. But then due to $\widetilde{Y} \cap (\widetilde{Z} \cup \{v\}) = \emptyset$, this node $w$ cannot be in both sets $\widetilde{Y}$ and $\widetilde{Z} \cup \{v\}$ at the same time. Thus $h \not\in \pa_{\cL}(\widetilde{Y}) \cap \pa_{\cL}(\widetilde{Z} \cup \{v\})$ as well. We conclude that condition (ii) of the LF-HTC is satisfied by the triple $(\widetilde{Y}, \widetilde{Z} , \widetilde{H})$ and therefore it satisfies the LF-HTC for $v\in V$. 
\end{proof}

\begin{proof}[Proof of Theorem 5.3]
  The proof works in the same way as the proof of Theorem $7$ in \citet{foygel2012halftrek}. We start by analyzing the complexity of the algorithm. 
  
  Observe that we run the ``inner'' algorithm (line $3$ to $15$) at most $|V|^2$ times. This can be seen by counting the maximal number of repetitions in line $1$. Another repetition is only done if a node was added to $S$ in the repetition before, otherwise the algorithm terminates. Thus after $|V|$ repetitions of line $1$ either all nodes were added to $S$ or the algorithm terminated before. By investigating line $2$ we see that in every pass we also iterate over at most $|V|$ nodes which yields the maximal number of $|V|^2$ runs of the inner algorithm.
  
  In the inner algorithm itself we iterate first through all sets $H \subseteq \cL_{\geq 4}  \subseteq \cL$ with cardinality at most $k$. The number of subsets of $\cL$ with cardinality at most $k$ is 
  $$
    \sum_{i=0}^{k} \binom{|\cL_{\geq 4}|}{i} = \mathcal{O}(|\cL|^k).
  $$
  In line $5$ we then iterate over all $Z \subseteq Z_a \subseteq V$ with $|Z|=|H|$. Similarly as before, we see that in the worst case these are $\mathcal{O}(|V|^k)$ iterations. Hence, we compute at most $\mathcal{O}(|V|^2 |\cL|^k |V|^k)$ maximum flows on a graph with at most $2(|V|+|\cL|)+2$ nodes and $4|V|+|\cL|+|D|$ edges and the same number $r$ of reciprocal edge pairs as in $D_V$. By \citet[Section 26]{cormen2009introduction} each maximum flow computation has complexity at most $\mathcal{O}((|V|+|\cL|+r)^3)$. Finally, note that the sets $\htr_H(U)$ for a subset $U \subseteq V$ can be found using breadth first search which has complexity $\mathcal{O}(|V|+|\cL|+|D|)$ by \citet[Section 22.2]{cormen2009introduction}. Finding parents and children of nodes is not of higher complexity. Since $|D| \leq |V|^2$, we conclude that the total complexity is $\mathcal{O}(|V|^{2+k} |\cL|^k (|V|+|\cL|+r)^3)$. 
  
  Next we show that the algorithm indeed determines LF-HTC-identifiability. Suppose that $\Gl$ is LF-HTC-identifiable. Then by Theorem 3.7 
  there is a total ordering $\prec$ on $V$ such that $w \prec v$ whenever $w \in Z_v \cup (Y_v \cap \htr_{H_v}(Z_v \cup \{v\}))$ where $(Y_v, Z_v, H_v) \in 2^{V\setminus\{v\}}\times 2^{V\setminus \{v\}}\times 2^\cL$ is a triple satisfying the LF-HTC with respect to $v$. Hence, if $\Gl$ is LF-HTC-identifiable, we might label the elements $\{v_1, \ldots, v_d\}=V$ such that $v_1 \prec v_2 \prec \cdots \prec v_d$. 
  
  Now we claim that after at most $k+1$ passes through the for loop in line $2$, all nodes $v_i$, $i \prec k$, have already been added to the solved nodes $S$. We prove this by induction. Suppose that all nodes $v_1, \ldots, v_{k-1} \in S$ and we are now testing the $k$-th node $v_k$. Let $(Y_{v_k}, Z_{v_k}, H_{v_k})$ be the triple satisfying the LF-HTC with respect to $v_k$. At one point, we will visit the correct set $H_{v_k} \in \cL_{\geq 4}$ in line $3$ due to Proposition 5.2. 
  If  $z \in Z_{v_k}$, then $z \prec v_k$ and therefore $z \in S$ already. Additionally, $z \in \ch(H_{v_k})$ and $z \not\in \{v_k\} \cup \pa_V(v_k)$ by definition of the LF-HTC. Thus, we will visit the correct set $Z_{v_k} \subseteq Z_a$ in line $5$. Now take any $y \in Y_{v_k}$. By definition of the LF-HTC, we have that $y \not\in Z_{v_k} \cup \{v_k\} \cup \ch(\pa(Z_{z_k} \cup \{v_k\}) \setminus H_{v_k})$. Moreover, if $y \in \htr_{H_{v_k}}(Z_{v_k} \cup \{v_k\})$, then $y \prec v_k$ and thus $y \in S$, which means $y \in A$. If instead $y \not\in \htr_{H_{v_k}}(Z_{v_k} \cup \{v_k\})$, then $y \in A$ by definition of $A$. Therefore, $Y_{v_k} \subseteq A$ and by Theorem 5.1 
  we will add $v_k$ to $S$. By induction, we obtain that $S=V$ after at most $|V|$ repetitions of line $2$ to $16$.

  Conversely, suppose the algorithm finds $S=V$, and fix a node $v \in V$. It remains to show that there is a triple $(Y_v, Z_v, H_v) \in 2^{V\setminus\{v\}}\times 2^{V\setminus \{v\}}\times 2^\cL$ such that all nodes $w \in Z_v \cup (Y_v \cap \htr_{H_v}(Z_v \cup \{v\}))$ were added to $S$ in the steps before. When $v$ was added to $S$, there must have been sets  $H_v \subseteq \cL_{\geq 4}$ and $Z_v \subseteq (S \cap  \ch(H_v)) \setminus (\{v\} \cup \pa_V(v))$ with $|Z|=|H|$ such that $\texttt{MaxFlow}(\Gf(v, A, Z_v)) = |\pa_V(v)| + |Z_v|$. By Theorem 5.1, 
  this means that there is a set $Y_v \subseteq A$ such that the triple $(Y_v, Z_v, H_v)$ satisfies the LF-HTC with respect to $v$. By construction, $Z_v \subseteq S$ at this stage of the algorithm. Moreover, we have for all $w \in A$ that either $w \in S$ already or $w \not\in \htr_{H_v}(Z_v \cup \{v\})$. Thus, we have as well that $Y_v \cap\htr_{H_v}(Z_v \cup \{v\}) \subseteq S$ at this stage of the algorithm. Applying this reasoning to all $v \in V$, we see that $\Gl$ is LF-HTC-identifiable.
  
\end{proof}

\section{NP-Hardness of the LF-HTC} \label{sec:NP-completeness}

In this section we show that the task of deciding \texttt{LF-HTC}$(\Gl,v)$ for unrestricted graphs is NP-hard. That is, it is at least as hard as the hardest problems in the NP-complexity class. Formally, we have to show that every problem in NP is reducible to \texttt{LF-HTC}$(\Gl,v)$ in polynomial time. Fortunately, it is enough to show that one arbitrary problem that is known to be NP-hard is reducible to \texttt{LF-HTC}$(\Gl,v)$ in polynomial time.  For this purpose, we choose the Boolean satisfiability problem in conjunctive normal form (\texttt{CNFSAT}). This is the problem of determining whether a Boolean expression in conjunctive normal form is satisfiable. That is, suppose we have Boolean variables $\{x_1,\dots,x_n\}$, and let
  \begin{align*}
    C = C_1\wedge \dots \wedge C_M = (\ell^1_1\vee \dots \vee l^1_{m_1}) \wedge \dots \wedge (\ell^M_1 \vee \dots \vee l^M_{m_M})
  \end{align*}
  where $\ell^i_j \in\{x_1,\dots,x_n,\neg x_1,\dots,\neg x_n\}$ for all $1\leq i\leq M$ and $1\leq j\leq m_i$. We call the elements of $\{x_1,\dots,x_n,\neg x_1,\dots,\neg x_n\}$ \emph{literals} and $\neg x_i$ the \emph{negation} of $x_i$. Then \texttt{CNFSAT}($\{x_1,\dots,x_n\},C)$ is the problem of determining if there exist assignments of \texttt{True} and \texttt{False} to each $x_i$ such that, under this assignment, $C$ is $\texttt{True}$, i.e., satisfied.

\begin{thm} \label{thm:np-complete}
  There exists a polynomial time reduction from \texttt{CNFSAT} to \texttt{LF-HTC}$(\Gl,v)$ so that \texttt{LF-HTC}$(\Gl,v)$ is NP-hard.
\end{thm}

\begin{proof}
  To see that there is a polynomial time reduction from \texttt{CNFSAT} to \texttt{LF-HTC}$(\Gl,v)$, let $X=\{x_1,\dots,x_n\}$ and let $C$ be as above. We now construct a latent-factor graph $G^\cL$ and show that solving \texttt{LF-HTC}$(\Gl,v)$ in $G^\cL$ solves \texttt{CNFSAT}($X,C$).  Our construction initializes $G^\cL=(V \cup \cL,D)$ to be empty. In the following when we add nodes to $G^\cL$ they will implicitly be added to $V$ unless they are labeled $h_{*}$ for some index $*$, in which case they are to be added to $\cL$.

  Begin by adding to the graph the nodes $v,w_1,\dots,w_M,h_{w_1v},\dots,h_{w_mv}$, the edges $w_i\to v$, and the edges $w_i \leftarrow h_{w_iv} \to v$ for all $1\leq i\leq M$. The $w_i$ will correspond to the $M$ disjunctive clauses in $C$. Now for the $i$th Boolean variable $x_i$, let $A_i$ be the number of times $x_i$ (in non-negated form) appears in $C$, and let $B_i$ be the number of times $\neg x_i$ appears in $C$. Then add to the graph
  \begin{enumerate}[(i)]
  \item the nodes $u_{i1},\ \dots,\ u_{iA_i},\quad \ol{u}_{i1},\ \dots,\ \ol{u}_{iB_i},\quad  u_i,\ \ol{u}_i, \quad h_{i},h_{\ol{i}}, \quad \text{and}\quad q_i$,

  \item $u_{ij}\to w_k$ if the $j$-th appearance, from the left, of $x_i$ (in non-negated form) in $C$ occurs in the $k$-th disjunctive clause of $C$,
  \item $\ol{u}_{ij}\to w_k$ if the $j$-th appearance, from the left, of $\neg x_i$ in $C$ occurs in the $k$-th disjunctive clause of $C$,
  \item $h_{i}\to a$ for each $a\in\{u_{i1},\ \dots,\ u_{iA_i}, u_i, q_i, v\}$,
  \item $h_{\ol{i}}\to a$ for each $a\in\{\ol{u}_{i1},\ \dots,\ \ol{u}_{iB_i}, \ol{u}_i, q_i, v\}$,
  \item a node $h_{ab}=h_{ba}$ and edges $a\leftarrow h_{ab} \rightarrow b$ for each pair of variables $a,b\in \{u_i,u_{i1},\dots,u_{iA_i}\}$, and
  \item a node $h_{ab}=h_{ba}$ and edges $a\leftarrow h_{ab} \rightarrow b$ for each pair of variables $a,b\in \{\ol{u}_i,\ol{u}_{i1},\dots,\ol{u}_{iB_i}\}$.
  \end{enumerate}

      \begin{figure}[t]
    \centering
    \tikzset{
      every node/.style={circle, inner sep=0.3mm, minimum size=0.45cm, draw, thick, black, fill=white, text=black},
      every path/.style={thick}
    }
    \resizebox{\linewidth}{!}{%
    \begin{tikzpicture}[align=center,node distance=2.2cm]
      \node [fill=lightgray] (o1) [] {$h_{1}$};
      \node [] (u1) [below of= o1]    {$u_1$};
      \node [] (u11) [below left of= u1]    {$u_{11}$};
      \node [] (u12) [below right of= u1]    {$u_{12}$};
      \node [] (q1) [right of= u1]    {$q_{1}$};
      
      \node [] (nu11) [below right of= q1]    {$\ol{u}_{11}$};
      \node [] (nu1) [above right of= nu11]    {$\ol{u}_1$};
      \node [fill=lightgray] (on1) [above of= nu1]    {$h_{\ol{1}}$};
      \node [] (nu12) [below right of= nu1]    {$\ol{u}_{12}$};
      
      \draw[red, dashed] [-latex] (o1) edge (u1);
      \draw[red, dashed] [-latex] (o1) edge (u11);
      \draw[red, dashed] [-latex] (o1) edge (u12);
      \draw[red, dashed] [-latex] (o1) edge (q1);
      \draw[red, dashed] [latex-latex] (u1) edge (u11);
      \draw[red, dashed] [latex-latex] (u1) edge (u12);
      \draw[red, dashed] [latex-latex] (u11) edge (u12);

      \draw[red, dashed] [-latex] (on1) edge (q1);
      \draw[red, dashed] [-latex] (on1) edge (nu11);
      \draw[red, dashed] [-latex] (on1) edge (nu12);
      \draw[red, dashed] [-latex] (on1) edge (nu1);
      \draw[red, dashed] [latex-latex] (nu1) edge (nu12);
      \draw[red, dashed] [latex-latex] (nu11) edge (nu1);
      \draw[red, dashed] [latex-latex] (nu11) edge (nu12);

      \node [] (w1) [below of= u12]    {$w_1$};
      \node [] (w2) [right of= w1]    {$w_2$};
      \node [] (w3) [right of= w2]    {$w_3$};
      \node [] (v) [below of= w2]    {$v$};

      \draw[blue] [-latex] (w1) edge (v);
      \draw[blue] [-latex] (w2) edge (v);
      \draw[blue] [-latex] (w3) edge (v);
      \draw[red, dashed] [latex-latex, bend left] (w1) edge (v);
      \draw[red, dashed] [latex-latex, bend left] (w2) edge (v);
      \draw[red, dashed] [latex-latex, bend left] (w3) edge (v);

      \draw[blue] [-latex] (u11) edge (w1);
      \draw[blue] [-latex] (u12) edge (w2);
      \draw[blue] [-latex] (nu11) edge (w1);
      \draw[blue] [-latex] (nu12) edge (w3);

      \node [] (u21) [right of= nu12]    {$u_{21}$};
      \node [] (u2) [above right of= u21]    {$u_2$};
      \node [fill=lightgray] (o2) [above of= u2]    {$h_{2}$};
      \node [] (u22) [below right of= u2]    {$u_{22}$};
      \node [] (q2) [right of= o2]    {$q_2$};
       
      \node [fill=lightgray] (no2) [right of= q2]    {$h_{\ol{2}}$};
      \node [] (nu2) [below of= no2]    {$\ol{u}_2$};

      \draw[red, dashed] [-latex] (o2) edge (q2);
      \draw[red, dashed] [-latex] (o2) edge (u2);
      \draw[red, dashed] [-latex] (o2) edge (u21);
      \draw[red, dashed] [-latex] (o2) edge (u22);
      \draw[red, dashed] [latex-latex] (u2) edge (u21);
      \draw[red, dashed] [latex-latex] (u2) edge (u22);
      \draw[red, dashed] [latex-latex] (u21) edge (u22);

      \draw[red, dashed] [-latex] (no2) edge (nu2);
      \draw[red, dashed] [-latex] (no2) edge (q2);
      \draw[red, dashed] [-latex] (no2) to [out=290,in=90] ($(nu2) + (0.5, 0)$) to [out=270,in=-10] (v);

      \draw[blue] [-latex] (u21) edge (w2);
      \draw[blue] [-latex] (u22) edge (w3);

      \draw[red, dashed] [-latex] (o1) to [out=180,in=90] ($(u11)-(0.8,0)$) to [out=270, in=160] (v);
      \draw[red, dashed] [-latex] (on1) to [out=150,in=0] ($(o1)-(0,-1)$) to [out=180, in=90] ($(u11)-(1,0)$) to [out=270,in=180] (v);
      \draw[red, dashed] [-latex] (o2) to [out=-20,in=90] ($(u22)+(1,0)$) to [out=270, in=0] (v);
      
    \end{tikzpicture}
    } 

    \caption{The graph $G^\cL$ corresponding to the \texttt{CNFSAT} problem with Boolean expression $C=(x_1\vee \neg x_1)\wedge (x_1\vee x_2) \wedge (\neg x_1\vee x_2)$. To not clutter the graph, a red bidirected edge (dashed) corresponds to a latent factor that has only arrows pointing to the two endpoints of the edge, e.g. the red bidirected edge $w_1\leftrightarrow v$ corresponds to $w_1\leftarrow h_{w_1v} \to v$.} \label{fig:cfnsat-to-graph}
  \end{figure}
  
  An example of a graph $G^\cL$ corresponding to a Boolean expression can be found in Figure \ref{fig:cfnsat-to-graph}. Now that we have constructed $G^\cL$ we claim that every triple $(Y,Z,H)$ satisfying the LF-HTC for $v \in V$ in $\Gl$  corresponds to an assignment to $X$ such that $C$ is satisfied under this assignment and vice versa.\\

  We will now start with the more complicated direction.  Suppose that there is a triple $(Y,Z,H)$ satisfying the LF-HTC for $v$ in $G^\cL$. That is, there exists a latent-factor half-trek system $\Pi$ from $Y$ to $Z\cup \pa_V(v)$ satisfying the appropriate LF-HTC conditions.  \\

  \textbf{Claim $1$:} No $w_i$ is an element of $Y$. \\
  
  Suppose for contradiction that some $w_i\in Y$. Since there exists a node $h_{w_iv}$ whose only edges are $w_i\leftarrow h_{w_iv} \rightarrow v$,  condition (ii) of the LF-HTC implies that $h_{w_iv}$ must be in $H$. But then condition (iii) implies that there must be some $z\in \text{ch}(h_{w_iv}) \cap Z$ for which the latent-factor half-trek $y\leftarrow h_{w_iv} \to z$ is in $\Pi$. By $\text{ch}(h_{w_iv}) = \{w_i,v\}$ we have a contradiction since if $z=v$ we would have $v\in Z$, and if $z=w_i$ we have that $w_i\in Y\cap Z$ so that $Y\cap Z\not=\emptyset$. Hence there is no $w_i \in Y$. \\

  \textbf{Claim $2$:} If $Y\cap\{u_i,\ u_{i1},\dots,u_{iA_i}\}\not=\emptyset$, then $Y\cap\{\ol{u}_i,\ \ol{u}_{i1},\dots,\ol{u}_{iB_i}\}=\emptyset$. \\
  
  Suppose that $u\in Y\cap\{u_i,\ u_{i1},\dots,u_{iA_i}\}$. Since $\text{ch}(h_i) = \{v, u_i, q_i, u_{i1},\dots,u_{iA_i}\}$, it follows from condition (ii) of the LF-HTC that $h_i\in H$. Hence, by condition (iii), it must be the case that there is some $y,z\in\{q_i,u_i,u_{i1},\dots,u_{iA_i}\}$ with $y\not=z$ such $y\in Y$ and $z\in Z$ and the latent-factor half-trek $y\leftarrow h_i\to z \in \Pi$. There are two cases.

  \underline{Case 1:} $z\not=q_i$. We must have that $z\in \{u_i,u_{i1},\dots,u_{iA_i}\}\setminus\{u\}$. Recall, in this case, that there exists $h_{uz}$ whose only children are $u$ and $z$. By a similar argument as in claim $1$, it follows that $u \leftarrow h_{uz}\to z$ must also be in $\Pi$ which contradicts the fact that $\Pi$ must have no sided intersection (since $z$ is already in the right-hand side of the latent-factor half-trek $u\leftarrow h_{i}\to z$).

  \underline{Case 2:} $z = q_i$. In this case we must have that the latent-factor half-trek $y \leftarrow h_i \to q_i\in \Pi$ for some $y\in\{u_i,u_{i1},\dots,u_{iA_i}\}$. Now, for essentially identical reasons as above, if we have $\ol{u}\in Y\cap \{\ol{u}_i,\ \ol{u}_{i1},\dots,\ol{u}_{iB_i}\}$,  there is some $\ol{y}\in \{\ol{u}_i,\ol{u}_{i1},\dots,\ol{u}_{iB_i}\}\setminus\{\ol{u}\}$ such that $\ol{y} \leftarrow h_{\ol{y}q_i} \to q_i$ is in $\Pi$. This contradicts the fact that $\Pi$ has no sided intersection (since $q_i$ is in the right part of two latent-factor half-treks) and thus we must have that $Y\cap \{\ol{u}_i,\ol{u}_{i1},\dots,\ol{u}_{iB_i}\}=\emptyset$.

  Since the first case results in a contradiction we must be in the second case with $Y\cap \{\ol{u}_i,\ol{u}_{i1},\dots,\ol{u}_{iB_i}\}=\emptyset$. \\

  \textbf{Claim $3$:} If $Y\cap\{\ol{u}_i,\ \ol{u}_{i1},\dots,\ol{u}_{iB_i}\}\not=\emptyset$, then $Y\cap\{u_i,\ u_{i1},\dots,u_{iA_i}\}=\emptyset$. \\

  The claim follows by symmetry from claim $2$. \\

  We can now show that our given triple $(Y,Z,H)$ satisfying the LF-HTC for $v \in V$ in $\Gl$  corresponds to an assignment to $X$ such that $C$ is satisfied under this assignment. For each $1\leq i \leq M$, assign $x_i$ to be \texttt{True} if $Y\cap\{u_i,\ u_{i1},\dots,u_{iA_i}\}\not=\emptyset$ and \texttt{False} otherwise. To see that this satisfies $C$ consider the $i$-th disjunctive clause $C_i=(\ell^i_1 \vee \dots \vee l^i_{m_i})$ of $C$. Since $(Y,Z,H)$ satisfies the LF-HTC for $v$, there must exist some $y \in Y$ such that there is a latent-factor half-trek from $y$ to $w_i$ in $\Pi$. By claim $1$ and the proof of claim $2$, we must have that $y\in \{u_j, u_{j1},\dots,u_{jA_j}\}$ or $y\in \{\ol{u}_j,\ol{u}_{j1},\dots,\ol{u}_{jB_j}\}$ for some $j$. Suppose that $y\in \{u_j, u_{j1},\dots,u_{jA_j}\}$. Then, since there must be a half-trek from $y$ to $w_i$, we have, by the construction of $G^\cL$ that $x_j$ appears (in non-negated form) in $C_i$. Since $\{u_j, u_{j1},\dots,u_{jA_j}\}\cap Y\not=\emptyset$, we must have set $x_j$ to be $\texttt{True}$ and thus $C_i$ is satisfied. If instead  $y\in \{\ol{u}_j, \ol{u}_{j1},\dots,\ol{u}_{jB_j}\}$, it follows, by the same logic, that $\neg x_j$ must appear in $C_i$ and that we set $x_j$ to be \texttt{False}, so again $C_j$ is satisfied. As $j$ was arbitrary,  $C$ is satisfied by the assignment. \\

  Now we wish to show the opposite direction, namely, that if there is  an assignment to $X$ such that $C$ is satisfied under this assignment, then there must also be a set $(Y,Z,H)$ satisfying the LF-HTC for $v$. Suppose we have assigned \texttt{True} and \texttt{False} to the $x_i$, so that $C$ is satisfied. Let $1\leq i\leq M$. For the $i$th disjunctive clause $C_i$ in $C$, let $l_k$ be the first literal in the clause, which evaluates to \texttt{True} (there must be at least one such literal since $C=\texttt{True}$ implies $C_i=\texttt{True}$). Now $l_k$ must equal $x_j$ or $\neg x_j$ for some $j$. Suppose that $l_k=x_j$. Then there exists a unique $1\leq \ell\leq A_j$ such that edge $u_{j\ell}\to w_i$ is in the graph. If $l_k=\neg x_j$ then, similarly, there exists a unique $1\leq \ell \leq B_j$ such that there exists an edge $\ol{u}_{j\ell}\to w_i$ in the graph. In either case, denote $u_{j\ell}$ or $\ol{u}_{j\ell}$ as $y_i$.

  Now for $M+1\leq i\leq M+n$, let $y_i = u_i$ if $x_i$ is \texttt{True}, and $y_u=\ol{u}_i$ otherwise. Let $Z = \{q_1,\dots,q_n\}$, $L = \{\tilde{h}_1,\dots,\tilde{h}_n\}$ where $\tilde{h}_i=h_i$ if $x_i=\texttt{True}$ and $\tilde{h}_i=h_{\ol{i}}$ if otherwise, and $Y = \{y_1,\dots,y_{M+n}\}$. By our construction, it holds that

  \begin{enumerate}[(i)]
  \item $|Y|=|\pa_V(v)| + |H|$,
  \item $Y\cap (Z\cup \{v\})=\emptyset$ and if $y\in Y$ and $v$ (or $z\in Z$) are children of the same latent factor, then that latent factor is some $h_i$ for which $x_i=\texttt{True}$ or some $h_{\ol{i}}$ for which $x_i=\texttt{False}$, and
  \item the set of latent-factor half-treks $\Pi$ with elements
    \begin{align*}
      &y_i \to w_i \quad\text{ for $1\leq i\leq M$}, \\
      &y_{M+i} \leftarrow h_{i} \to z_i=q_i \quad \text{ for $1\leq i\leq n$ if $x_i$ is \texttt{True}, and}\\
      &y_{M+i} \leftarrow h_{\ol{i}}\to z_i=q_i \quad \text{ for $1\leq i\leq n$ if $x_i$ is \texttt{False}}
    \end{align*}
    forms a latent-factor half-trek system from $Y$ to $\pa_V(v)\cup Z$ for which the half-trek to each $z_i$ is of the form $y_{M+i}\leftarrow \tilde{h}_{i} \to z_i$ and if $y_i$ has a common latent parent with $v$ (or $z\in Z$) then the latent parent must correspond to some $\tilde{h}_j\in H$ and we have that $y_{M+j}\leftarrow h_{j}\to q_j \in \Pi$.
  \end{enumerate}

  Note that the above three conditions immediately imply that $(Y,Z,H)$ satisfies the LF-HTC for $v$. We have thus shown that \texttt{CNFSAT} reduces to \texttt{LF-HTC}$(\Gl,v)$ in polynomial time.
\end{proof}

\section{Algebraic Techniques for Determining Identifiability} \label{sec:algebraic-techniques}
As discussed in \citet{garcia2010identifying}, rational identifiability may be decided by techniques from computational algebraic geometry.  For the original half-trek criterion, \citet{foygel2012halftrek} provide an effective algorithm to perform the necessary computations.  In this section we show how their approach may be generalized to cover the latent-factor setup from this paper.

Consider a slightly more general setting than before, i.e., let $S \subseteq \mathbb{R}^m$ be an open semialgebraic set, and let
\begin{align*}
    \tau : S &\longrightarrow \textit{PD}(d) \\
    \Delta &\longmapsto \tau(\Delta)
\end{align*}
be a polynomial map that parametrizes the cone of latent covariance matrices $\im(\tau)$. Together with a directed graph on the observed nodes $G^V=(V, D_V)$ with $V=|d|$, the cone of latent covariance matrices $\im(\tau)$ postulates a covariance model. 

\begin{defn} \label{def:model-tau}
    The covariance model given by a tuple $(G^V, \im(\tau))$, consisting of a directed graph $G^V=(V, D_V)$ with $V=|d|$ and a cone of latent covariance matrices $\im(\tau)$, is given by the family of all covariance matrices
    $$
        \Sigma = (I_d - \Lambda)^{-\top}\Omega(I_d-\Lambda)^{-1}
    $$
    for $\Lambda \in \mathbb{R}^{D_V}_{\mathrm{reg}}$ and $\Omega \in \im(\tau)$.
\end{defn}

A covariance matrix $\Sigma \in \textit{PD}(d)$ is in the covariance model given by a tuple $(G^V, \im(\tau))$ if and only if $\Sigma = \phi(\Lambda, \tau(\Delta))$ for $\Lambda \in \mathbb{R}^{D_V}_{\mathrm{reg}}$ and $\Delta \in S$ where the parametrization map $\phi$ is given by
\begin{align*}
\phi : \mathbb{R}^{D_V}_{\mathrm{reg}} \times \textit{PD}(d) &\longrightarrow \textit{PD}(d) \\
    (\Lambda, \Omega) &\longmapsto (I_d - \Lambda)^{-\top}\Omega(I_d-\Lambda)^{-1}.
\end{align*}

If we let $S=\mathbb{R}^{D_{\cL V}} \times \mathrm{diag}^{+}_d$ and $\tau: (\Gamma, \Od) \mapsto \Od + \Gamma^{\top} \Gamma$, then Definition \ref{def:model-tau} coincides with  Definition 2.1
of a covariance model given by a latent-factor graph.

In what follows, let $\lambda = \{\lambda_{ij}: i \rightarrow j \in D_V\}$ be variables representing the non-zero entries of $\Lambda \in \mathbb{R}^{D_V}_{\mathrm{reg}}$, and let $\omega = \{\omega_{ij}: 1 \leq i \leq j \leq d\}$ be variables representing the entries of $\Omega \in \textit{PD}(d)$. Let $d(\lambda) \in \mathbb{R}[\lambda]$ be the polynomial defined by $\det(I_d - \Lambda)$ for $\Lambda \in \mathbb{R}^{D_V}$. Now observe that, for $(\Lambda, \Omega) \in \mathbb{R}^{D_V}_{\mathrm{reg}} \times \textit{PD}(d)$, we may write the $ij$-th entry of the matrix $\phi(\Lambda, \Omega)$ as a rational function
$$
    \phi_{ij}(\Lambda, \Omega) = \frac{\widetilde{\phi}_{ij}(\lambda, \omega)}{d(\lambda)^{2}}
$$
with $\widetilde{\phi}_{ij}(\lambda, \omega) \in \mathbb{R}[\lambda, \omega]$ due to Cramer's rule. Furthermore, we write $\sigma = \{\sigma_{ij}: 1 \leq i \leq j \leq d\}$ and $\delta = \{\delta_i: i=1,\ldots, m\}$ for variables representing the entries of $\Sigma \in \textit{PD}(d)$ and $\Delta \in S$, respectively. Consider the polynomial ring $\mathbb{R}[\lambda, \sigma, \delta, \xi]$ with one additional variable $\xi$.  Then the vanishing ideal of the graph of the parametrization in Definition \ref{def:model-tau} is given by 
$$
    \mathcal{J} = \langle \{\sigma_{ij} d(\lambda)^2 - \widetilde{\phi}_{ij}(\lambda, \tau(\delta)), 1 \leq i \leq j \leq d\} \cup \{ \xi d(\lambda) - 1 \}\rangle \subseteq \mathbb{R}[\lambda, \sigma, \delta, \xi].
$$
The additional variable $\xi$ and the polynomial $\xi d(\lambda) - 1$ are needed to ensure that $d(\lambda)$ never vanishes. Eliminating the variables $\lambda$, $\delta$ and $\xi$, we get the vanishing ideal of the image of $\Theta = \mathbb{R}^{D_V}_{\mathrm{reg}} \times \im(\tau)$ under the parametrization $\phi$, in formulas,
$$
    I(\phi(\Theta)) = \mathcal{J} \cap \mathbb{R}[\sigma].
$$
Nevertheless, for the purpose of identifying direct causal effects, we are interested in an ideal where $\lambda$ is not eliminated, i.e., we are interested in
$$
    \mathcal{I} =  \mathcal{J} \cap \mathbb{R}[\lambda, \sigma].
$$
By definition, this ideal consists exactly of those polynomials $f(\lambda, \sigma) \in \mathbb{R}[\lambda, \sigma]$ such that $f(\Lambda, \phi(\Lambda, \Omega)) = 0$ for all $(\Lambda, \Omega) \in \Theta$.

\begin{prop} \label{prop:check-identifiability}
The parameter $\lambda_{ij}$ is rationally identifiable if and only if $\mathcal{I}$ contains an element of the form $a(\sigma) \lambda_{ij} - b(\sigma)$ with $a,b \in \mathbb{R}[\sigma]$ and $a\not\in I(\phi(\Theta))$.
\end{prop}
\begin{proof}
The proof is similar to the proofs of Lemma 7 in \citet{foygel2012halftrek-supp} and Proposition 3 in \citet{garcia2010identifying}. For completeness, we give the full proof in our notation. Suppose that the parameter $\lambda_{ij}$ is rationally identifiable. By definition, there is a rational function $b(\sigma)/a(\sigma) \in \mathbb{R}(\sigma)$ such that 
$$
    \frac{b(\phi(\Lambda, \Omega))}{a(\phi(\Lambda, \Omega))} = \lambda_{ij}.
$$
for all $(\Lambda, \Omega) \in \Theta \setminus A$, where $A$ is a proper algebraic subset of the Zariski closure of $\Theta$. In particular, outside $A$ we must have that $a(\phi(\Lambda, \Omega)) \neq 0$ and therefore it must be the case that $a \not\in I(\phi(\Theta))$. On the other hand, it is clear that the polynomial $a(\sigma) \lambda_{ij} - b(\sigma)$ is a member of $\mathcal{I}$ since this polynomial vanishes if we substitute $\sigma$ by $\phi(\Lambda, \Omega)$ for any $(\Lambda, \Omega) \in \Theta \setminus A$. 

Conversely, suppose that $a$ and $b$ satisfy the given conditions. Since $a \not\in  I(\phi(\Theta))$, we have $a(\phi(\Lambda, \Omega)) \neq 0$ for all $(\Lambda, \Omega) \in \Theta \setminus A$, where $A$ is a proper algebraic subset of the Zariski closure of $\Theta$. But then $b/a$ is a rational function identifying $\lambda_{ij}$ from $\sigma$ outside the proper algebraic subset $A$.
\end{proof}

For checking rational identifiability one has to check the membership of polynomials of the form $a(\sigma) \lambda_{ij} - b(\sigma)$ in the ideal $\mathcal{I}$. This can be achieved by computation of a Gr\"obner basis of $\mathcal{I}$ with eliminating term order using Buchberger's algorithm; see \citet{garcia2010identifying}. The Gr\"obner basis computation can be very challenging in terms of running times, even for small graphs. One reason is that the input polynomials to the computation $\widetilde{\phi}_{ij}(\lambda, \tau(\delta)), 1 \leq i \leq j \leq n$, may already have large degree. As mentioned in \citet{foygel2012halftrek-supp}, it is easy to construct graphs where the bit-size of those polynomials is already exponential in the size of the graphs.
By Definition \ref{def:model-tau}, we have the equation $\Sigma = (I_d - \Lambda)^{-\top}\tau(\Delta)(I_d-\Lambda)^{-1}$. Since the matrix $(I_d - \Lambda)$ is required to be invertible, this is equivalent to the equation
\begin{equation} \label{eq:psi}
    (I_d - \Lambda)^{\top}\Sigma(I_d-\Lambda) = \tau(\Delta).
\end{equation}
Clearly, the entries of the matrix on the left-hand side are cubic, i.e., the maximal degree of the involved polynomials in $\sigma$ and $\lambda$ is $3$. We suggest to exploit this fact instead of computing the Gr\"obner basis for $\mathcal{I}$ directly. The idea is to compute a generating set of the vanishing ideal of the cone of latent covariance matrices $\im(\tau)$ and then to plug-in the polynomials from the left-hand side. The resulting polynomials then indeed define the same ideal $\mathcal{I}$ and may be much smaller in size. Therefore, the Gr\"obner basis computation may be faster. This is proved in Proposition \ref{prop:idealI} below, but we need to introduce some more notation beforehand. 

As usual, we denote $I(\im(\tau)) = \{f \in \mathbb{R}[\omega] : f(\Omega) = 0 \textrm{ for all } \Omega \in \im(\tau)\}$ for the vanishing ideal of $\im(\tau)$. We will also need the map corresponding to Equation \eqref{eq:psi}, i.e., 
\begin{align*}
\psi : \mathbb{R}^{D_V}_{\mathrm{reg}} \times \textit{PD}(d) &\longrightarrow \textit{PD}(d) \\
    (\Lambda, \Sigma) &\longmapsto (I_d - \Lambda)^{\top}\Sigma(I_d-\Lambda).
\end{align*}
Note that $\psi$ may be interpreted as an ``inverse'' of $\phi$ in the sense that $\psi(\Lambda, \phi(\Lambda, \Omega)) = \Omega$ and $\phi(\Lambda, \psi(\Lambda, \Sigma)) = \Sigma$. Since $\psi$ and $\tau$ are polynomial functions by definition, we write under abuse of notation $\psi(\lambda, \sigma)$ and $\tau(\delta)$ for the collection of polynomials they define. Similarly, we write $\widetilde{\phi}(\lambda, \omega)$ for the collection of polynomial functions $\widetilde{\phi}_{ij}(\lambda, \omega)$ for $1 \leq i \leq j \leq d$. 

\begin{prop} \label{prop:idealI}
Let $A_S = \{h \circ \psi \in \mathbb{R}[\lambda, \sigma]: h \in I(\im(\tau))\}$. Then 
$$
\mathcal{I} = \langle A_S, \xi d(\lambda) - 1 \rangle \cap \mathbb{R}[\lambda, \sigma].
$$
\end{prop}

\begin{proof} 
We begin by showing $A_S \subseteq \mathcal{I}$. Thus let $f(\lambda, \sigma) \in A_S$. By definition of $A_S$, there is $h \in I(\im(\tau))$ such that $f=h \circ \psi$. Hence, for any point $(\Lambda, \Omega) \in \Theta$, we have
$$
     f(\Lambda, \phi(\Lambda, \Omega)) = h(\psi(\Lambda,\phi(\Lambda, \Omega))) = h(\Omega) = 0
$$
since $\Omega \in \im(\tau)$. This means that $f \in \mathcal{I}$ and therefore $A_S \subseteq \mathcal{I}$. But this yields that $\langle A_S,  \xi d(\lambda) - 1 \rangle \subseteq \mathcal{J}$, and by the definition of $\mathcal{I}$ we conclude that $\langle A_S,  \xi d(\lambda) - 1 \rangle \cap \mathbb{R}[\lambda, \sigma] \subseteq \mathcal{I}$.

For the other inclusion, let $I(\im(\tau))=\langle h_1, \ldots, h_r\rangle \subseteq \mathbb{R}[\omega]$ and $f(\lambda, \sigma) \in \mathcal{I}$. Define the polynomial $g(\lambda, \omega, \xi) = f(\lambda,  \xi^2 \widetilde{\phi}(\lambda, \omega))$, which is an element of the polynomial ring $\mathbb{R}[\lambda, \omega, \xi]$. By the definition of $\mathcal{I}$, the polynomial $g$ becomes zero if we plug in $(\lambda, \omega, \xi) = (\Lambda, \Omega, d(\Lambda)^{-1})$ for any $(\Lambda, \Omega) \in \Theta$. Therefore, $g$ lies in the ideal $\langle h_1(\omega), \ldots, h_r(\omega), \xi d(\lambda) - 1 \rangle$ interpreted in the ring $\mathbb{R}[\lambda, \omega, \xi]$. Hence, we can write

$$
    g(\lambda, \omega, \xi) = \sum_{i=1}^r g_i(\lambda, \omega, \xi) h_i(\omega) + g_{r+1}(\lambda, \omega, \xi) (\xi d(\lambda) - 1)
$$
for suitable coefficient polynomials $g_i(\lambda, \omega, \xi) \in \mathbb{R}[\lambda, \omega, \xi]$. Plugging in $\psi(\lambda, \sigma)$ for $\omega$, we see that 
\begin{align*} \label{eq:function-g}
    g(\lambda, \psi(\lambda, \sigma), \xi) = &\sum_{i=1}^r g_i(\lambda, \psi(\lambda, \sigma), \xi ) h_i(\psi(\lambda, \sigma)) \\
    &+ g_{r+1}(\lambda, \psi(\lambda, \sigma), \xi) (\xi d(\lambda) - 1)
\end{align*}
is an element of $\langle A_S, \xi d(\lambda) - 1 \rangle$ since each polynomial $h_i(\psi(\lambda, \sigma)) \in A_S.$ Moreover, we have the equality
$$
g(\lambda, \psi(\lambda, \sigma), \xi) = f(\lambda, \xi^2 \widetilde{\phi}(\lambda,\psi(\lambda, \sigma))) = f(\lambda, \xi^2 d(\lambda)^2 \sigma)
$$
and thus $f(\lambda, \xi^2 d(\lambda)^2 \sigma) \in \langle A_S, \xi d(\lambda) - 1 \rangle$. But by the fact that $\xi d(\lambda) - 1 \in \langle A_S, \xi d(\lambda) - 1 \rangle$, it holds that the difference 
$$
   f(\lambda, \xi^2 d(\lambda)^2 \sigma) -  f(\lambda, \sigma) \in \langle A_S, \xi d(\lambda) - 1 \rangle
$$
and therefore it must be the case that the polynomial $f(\lambda, \sigma)$ itself is in the ideal $\langle A_S, \xi d(\lambda) - 1 \rangle$ since every ideal is an additive group. We conclude the proof by noting that $f(\lambda, \sigma)$ does not depend on $\xi$ which means that $f(\lambda, \sigma) \in \mathbb{R}[\lambda, \sigma]$ as well.
\end{proof}

Propositions \ref{prop:check-identifiability} and \ref{prop:idealI} yield Algorithm \ref{alg:check-algebraically} for checking rational identifiability of a covariance model given by $(G^V, \im(\tau))$. The proof of the correctness of the algorithm is identical to the proof of Algorithm $1$ in \citet{foygel2012halftrek-supp} and therefore omitted.

\begin{algorithm}[b]
\caption{Algebraically checking rational identifiability}
\begin{algorithmic}[1]
\STATE{Compute a Gr\"obner basis $\langle h_1, \ldots, h_r \rangle \subseteq \mathbb{R}[\omega]$  of the vanishing ideal $I(\im(\tau))$ using elimination theory.}
\STATE{Choose a block-monomial order $\geq$ on the monomials in the variables $\lambda, \sigma$ with $\lambda > \sigma$}.
\STATE{Let $\mathcal{I} = \langle h_1(\psi(\lambda, \sigma)), \ldots, h_r(\psi(\lambda, \sigma)), \xi d(\lambda) - 1 \rangle \cap \mathbb{R}[\lambda, \sigma]$ and compute the reduced Gr\"obner basis $T$ with respect to $\geq$ of the ideal $\mathcal{I}$.}
\STATE{The covariance model given by $(G^V, \im(\tau))$ is rationally identifiable if and only if for each $i \rightarrow j \in D_V$ the basis $T$ contains an element whose leading monomial equals a monomial in $\sigma$ times $\lambda_{ij}$.}
\end{algorithmic}
\label{alg:check-algebraically}
\end{algorithm}

With the reduced Gr\"obner basis obtained in step 3 of Algorithm~\ref{alg:check-algebraically}, it is not just possible to determine rational identifiability, but it is straightforward to modify the algorithm to check if a graph is generically finite-to-one \citep{garcia2010identifying}. 

Note that the computation of the ideal $\mathcal{I}$ requires the polynomials  $\psi(\lambda, \sigma)$. To speed-up the algorithm for large-scale computational experiments as in Section
6, 
it is advantageous to replace the variables $\sigma$ by numerical values obtained from the entries of a randomly chosen matrix $\Sigma$ in the model. Put differently, we randomly generate parameters $\Lambda_0 \in \mathbb{R}^{D_V}_{\mathrm{reg}}$ and $\Delta_0 \in S$ and then use the polynomials $\psi(\lambda, \phi(\Lambda_0, \tau(\Delta_0)))$ instead of $\psi(\lambda, \sigma)$. The Gr\"obner basis then readily yields the dimension and cardinality of the solution set. In practice, we generate $(\Lambda_0, \Delta_0)$ from prime numbers and we repeat the randomized calculation several times for each graph to avoid false conclusions from random draws yielding parameters $(\Lambda_0, \tau(\Delta_0))$ in special constellations.

\begin{exmp}
    Consider the latent-factor graph in  Figure \ref{fig:example-algebraic-techniques}. In this case, the parameter space $S$ is given by $S=\mathbb{R}^{D_{\cL V}} \times \mathrm{diag}^{+}_d$ and we have $\tau: (\Gamma, \Od) \mapsto \Od + \Gamma^{\top} \Gamma$. By implicitization, we find that the Gr\"obner basis of $I(\im(\tau))$ is given by the following list of polynomials:
    $$
        \omega_{12}\omega_{34} - \omega_{14}\omega_{23}, \ \omega_{13}\omega_{24} - \omega_{14}\omega_{23}, \ \omega_{15}, \ \omega_{16}, \ \omega_{25}, \ \omega_{26}, \ \omega_{35}, \ \omega_{36}.
    $$
    Now, we plug in the relevant polynomials $\psi(\lambda, \sigma)$, which for this graph are given by
    \begin{align*}
        \psi_{12}(\lambda, \sigma) &= -\lambda_{12}\sigma_{11} + \sigma_{12}, \\
        \psi_{13}(\lambda, \sigma) &= -\lambda_{23}\sigma_{12} + \sigma_{13}, \\
        \psi_{14}(\lambda, \sigma) &= \sigma_{14}, \\
        \psi_{15}(\lambda, \sigma) &= -\lambda_{45}\sigma_{14} + \sigma_{15}, \\
        \psi_{16}(\lambda, \sigma) &= -\lambda_{46}\sigma_{14} + \sigma_{16}, \\
        \psi_{23}(\lambda, \sigma) &= \lambda_{12}\lambda_{23}\sigma_{12} - \lambda_{12} \sigma_{13} - \lambda_{23} \sigma_{22} + \sigma_{23}, \\
        \psi_{24}(\lambda, \sigma) &= -\lambda_{12}\sigma_{14} + \sigma_{24}, \\
        \psi_{25}(\lambda, \sigma) &= \lambda_{12}\lambda_{45}\sigma_{14} - \lambda_{12} \sigma_{15} - \lambda_{45} \sigma_{24} + \sigma_{25}, \\
        \psi_{26}(\lambda, \sigma) &= \lambda_{12}\lambda_{46}\sigma_{14} - \lambda_{12} \sigma_{16} - \lambda_{46} \sigma_{24} + \sigma_{26}, \\
        \psi_{34}(\lambda, \sigma) &= -\lambda_{23}\sigma_{24} + \sigma_{34}, \\
        \psi_{35}(\lambda, \sigma) &= \lambda_{23}\lambda_{45}\sigma_{24} - \lambda_{23} \sigma_{25} - \lambda_{45} \sigma_{34} + \sigma_{35}, \\
        \psi_{36}(\lambda, \sigma) &= \lambda_{23}\lambda_{46}\sigma_{24} - \lambda_{23} \sigma_{26} - \lambda_{46} \sigma_{34} + \sigma_{36}.
    \end{align*}
    As in step 3 of Algorithm \ref{alg:check-algebraically}, we compute the reduced Gr\"obner basis $T$ of the ideal $\mathcal{I}$. Since $T$ contains the four polynomials
    \begin{align*}
        &\lambda_{12} \sigma_{11} \sigma_{12} \sigma_{24} \sigma_{34} - \lambda_{12} \sigma_{11} \sigma_{13} \sigma_{24}^2 - \lambda_{12} \sigma_{11} \sigma_{14} \sigma_{22} \sigma_{34} \\
        & \qquad + \lambda_{12} \sigma_{11} \sigma_{14} \sigma_{23} \sigma_{24} - \lambda_{12} \sigma_{12} \sigma_{14}^2 \sigma_{23} + \lambda_{12} \sigma_{13} \sigma_{14}^2 \sigma_{22} \\ & \qquad - \sigma_{12}^2 \sigma_{24} \sigma_{34} + \sigma_{12} \sigma_{13} \sigma_{24}^2 + \sigma_{12} \sigma_{14} \sigma_{22} \sigma_{34} - \sigma_{13} \sigma_{14} \sigma_{22} \sigma_{24}, \\
        &\lambda_{23} \sigma_{12} \sigma_{24} - \lambda_{23} \sigma_{14} \sigma_{22} - \sigma_{13} \sigma_{24} + \sigma_{14} \sigma_{23},\\
        &\lambda_{45} \sigma_{14} - \sigma_{15} \textrm{ and } \\
        &\lambda_{46} \sigma_{14} - \sigma_{16},
    \end{align*}
    we conclude that the latent-factor graph in Figure \ref{fig:example-algebraic-techniques} is rationally identifiable. It is in fact even LF-HTC-identifiable.
\end{exmp}

\begin{figure}[t]
    \centering
    \tikzset{
      every node/.style={circle, inner sep=0.3mm, minimum size=0.45cm, draw, thick, black, fill=white, text=black},
      every path/.style={thick}
    }
    \begin{tikzpicture}[align=center]
      \node[fill=lightgray] (h1) at (-1,1) {$h_{1}$};
      \node[fill=lightgray] (h2) at (1.5,1) {$h_{2}$};
      \node[] (1) at (-2.5,0) {1};
      \node[] (2) at (-1.5,0) {2};
      \node[] (3) at (-0.5,0) {3};
      \node[] (4) at (0.5,0) {4};
      \node[] (5) at (1.5,0) {5};
      \node[] (6) at (2.5,0) {6};
      
      \draw[red, dashed] [-latex] (h1) edge (1);
      \draw[red, dashed] [-latex] (h1) edge (2);
      \draw[red, dashed] [-latex] (h1) edge (3);
      \draw[red, dashed] [-latex] (h1) edge (4);
      \draw[red, dashed] [-latex] (h2) edge (4);
      \draw[red, dashed] [-latex] (h2) edge (5);
      \draw[red, dashed] [-latex] (h2) edge (6);
      
      \draw[blue] [-latex] (1) edge (2);
      \draw[blue] [-latex] (2) edge (3);
      \draw[blue] [-latex] (4) edge (5);
      \draw[blue] [-latex, bend right] (4) edge (6);
      
    \end{tikzpicture}
    \caption{LF-HTC-identifiable and therefore rationally identifiable.}
    \label{fig:example-algebraic-techniques}
\end{figure}
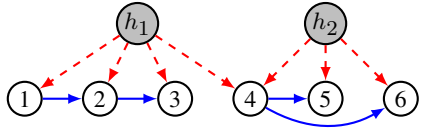
\end{appendix}

\bibliographystyle{imsart-nameyear} 
\bibliography{literature}       